\documentclass[12pt,leqno,headings]{amsart}
\usepackage[utf8]{inputenc}
\usepackage[T1]{fontenc}
\usepackage[english]{babel}
\usepackage{amsmath}
\usepackage{amsfonts}
\usepackage{amssymb}
\usepackage{amsthm}
\usepackage{amscd}
\usepackage{enumerate}
\usepackage[shortlabels]{enumitem}
\usepackage{fullpage}
\usepackage{mathtools}
\usepackage{mathdots}
\usepackage{mathrsfs}
\usepackage{nameref}
\usepackage[pdfauthor={Niels Lindner},
pdftitle={Hypersurfaces with defect}]{hyperref}
\usepackage{amsrefs}

\newcommand{\coker}{\operatorname{coker}}
\renewcommand{\dim}{\operatorname{dim}}

\newcommand{\Spec}{\operatorname{Spec}}

\newcommand{\tensor}{\otimes}
\newcommand{\rk}{\operatorname{rk}}

\newcommand{\Gr}{\operatorname{Gr}}

\newcommand{\rig}{\mathrm{rig}}
\newcommand{\dR}{\mathrm{dR}}
\newcommand{\KdR}{\mathrm{KdR}}
\newcommand{\et}{\mathrm{\text{\'et}}}

\newcommand{\sing}{\mathrm{sing}}
\newcommand{\an}{\mathrm{an}}

\newcommand{\Pic}{\operatorname{Pic}}
\newcommand{\Cl}{\operatorname{Cl}}
\newcommand{\NS}{\operatorname{NS}}

\newcommand{\characteristic}{\operatorname{char}}

\newtheorem{theorem}{Theorem}[section]
\newtheorem{proposition}[theorem]{Proposition}
\newtheorem{corollary}[theorem]{Corollary}
\newtheorem{lemma}[theorem]{Lemma}
\newtheorem{fact}[theorem]{Fact}

\theoremstyle{definition}

\newtheorem*{question}{Question}

\theoremstyle{remark}
\newtheorem*{remark}{Remark}
\newtheorem*{remarks}{Remarks}

\begin{document}

\title[Hypersurfaces with defect]{Hypersurfaces with defect}
\author{Niels Lindner}
\date{\today}

\begin{abstract}
 A projective hypersurface $X \subseteq \mathbb P^n$ has \textit{defect} if $h^i(X) \neq h^i(\mathbb P^n)$ for some $i \in \{n, \dots, 2n-2\}$ in a suitable cohomology theory. This occurs for example when $X \subseteq \mathbb P^4$ is not $\mathbb Q$-factorial. We show that in characteristic $0$, the Tjurina number of hypersurfaces with defect is large. For $X$ with mild singularities, there is a similar result in positive characteristic. As an application, we obtain a lower bound on the asymptotic density of hypersurfaces without defect over a finite field.
\end{abstract}

\maketitle

\section{Introduction}

Let $K$ be a field and let $n \geq 3$ be an integer. A projective hypersurface $X \subseteq \mathbb P^n_K$ is said to have \textit{defect} if
$$ h^i(X) \neq h^i(\mathbb P^n_K) \quad \text{ for some } i \in \{n, \dots, 2n-2\},$$
where $h^n$ denotes the $n$-th Betti number in a reasonable cohomology theory for $K$-varieties. Some prominent examples of such cohomology theories include:
\begin{itemize}
 \item singular, algebraic de Rham or Kähler-de Rham cohomology if $K$ is of characteristic zero,
 \item rigid cohomology if $K$ is a perfect field of positive characteristic,
 \item étale cohomology.
\end{itemize}

In any of these theories, a hypersurface with defect is necessarily singular. Moreover, it seems that defect forces the hypersurface to have ``many'' singularities compared to their degree: For example, an important class of hypersurfaces with defect is formed by non-factorial hypersurfaces $X \subseteq \mathbb P^4$ (see also Section \ref{S:factorial}). By a result of Cheltsov \cite{Cheltsov}, if such an $X$ has at most ordinary double points as singularities, then the singular locus consists of at least $(\deg(X) - 1)^2$ nodes.

Another family of hypersurfaces of defect in $\mathbb P^n$ is given by cones over smooth hypersurfaces in $\mathbb P^{n-1}$, see Corollary~\ref{C:Cone} and the subsequent remark. The vertex of the cone is a singularity with big Milnor number.

\vspace{\baselineskip}
The literature on defect (e.g., \cite{Dimca-Betti}, \cite{Rams}, \cite{Werner}) is mainly on hypersurfaces with at most ordinary double points as singularities and exclusively over the field of complex numbers. The aim of this paper is to generalize the philosophy ``defect implies many singularities'' to arbitrary projective hypersurfaces over arbitrary fields. In Section~\ref{S:char0}, we prove the following:
\begin{theorem}\label{T:Main0}
 Let $K$ be a field of characteristic zero. Suppose that $X \subseteq \mathbb P^n_K$, $n \geq 3$, is a hypersurface with defect in algebraic de Rham, Kähler-de Rham, singular or étale cohomology. Denote by $\tau(X)$ the global Tjurina number of $X$. Then
$$ \tau(X) \geq \frac{\deg(X)-n+1}{n^2+n+1}.$$
Moreover, if $X$ has at most weighted homogeneous singularities, then
$$ \tau(X) \geq \deg(X) - n + 1.$$
\end{theorem}
Of course, $\tau(X)$ will only be finite if $X$ has at most isolated singularities. The main ingredient in the proof is a close inspection of the algebraic de Rham cohomology of hypersurface complements in the spirit of Griffiths \cite{Griffiths-PeriodsI} and Dimca \cite{Dimca-Betti}.

The situation for positive characteristic fields is much more subtle. As explained in Subsection~\ref{SS:charp}, there are some obstructions to extending the proof of Theorem~\ref{T:Main0}. However, for hypersurfaces with very mild singularities, we can use a resolution of singularities approach similar to \cite{PRS} to prove:
\begin{theorem}\label{T:MainRes}
 Let $K$ be an algebraically closed field of characteristic $\neq 2$. Let $X \subseteq \mathbb P^n_K$ be a hypersurface with defect in étale or rigid cohomology. Suppose further that $X$ has a zero-dimensional singular locus $\Sigma = \Sigma_O \cup \Sigma_A$, where
 \begin{itemize}
  \item $\Sigma_O$ is formed by $x \in \Sigma$ being ordinary multiple points of multiplicity $m_x$ and
  \item $\Sigma_A$ consists of $x \in \Sigma$ which are singular points of type $A_{k_x}$.
 \end{itemize}
 Then
 $$\sum_{x \in \Sigma_O} m_x +  \sum_{x \in \Sigma_A} 2 \left\lceil \frac{k_x}{2} \right\rceil \geq \deg(X).$$ 
\end{theorem}
For details, see Section~\ref{S:res}. We conjecture that the theme ``defect implies many singularities'' should extend to arbitrary hypersurfaces in any positive characteristic.

\vspace{\baselineskip}
As an application of Theorem~\ref{T:MainRes}, we prove in Section~\ref{S:density} (Corollary~\ref{C:Density}) a lower bound on the density of hypersurfaces without defect over a finite field:
\begin{theorem}\label{T:Density} Let $q$ be an odd prime power. Then
 \begin{align*}
  \lim_{d \to \infty} \frac{\#\{f \in \mathbb F_q[x_0,\dots,x_n]_d \mid \{f = 0\} \subseteq \mathbb P^n_{\mathbb F_q} \text{ has no defect}\}}{\#\mathbb F_q[x_0,\dots,x_n]_d} \geq \frac{1}{\zeta_{\mathbb P^n}(n+3)} = \prod_{i=3}^{n+2} \left(1 - q^{-i} \right).
 \end{align*}
\end{theorem}
\noindent
In view of Theorem~\ref{T:Main0} and \cite{Lindner-Density}*{Corollary~5.9}, we believe that this limit is actually $1$.

\subsection*{Acknowledgements}

This article constitutes a part of my PhD project, which has been supported by IRTG 1800 \textit{Moduli and Automorphic Forms}. The final stage of research took place at Leibniz Universität Hannover. I wish to thank my supervisor Remke Kloosterman and the working groups in Hannover and Berlin.

\section{De Rham Cohomology}\label{S:char0}

\subsection{Two de Rham cohomology theories}

Let $K$ be a field of characteristic zero and let $X$ be a $K$-variety. Suppose that $X$ admits a closed embedding $X \hookrightarrow Y$ into some smooth $K$-variety $Y$. One can associate two related cohomology theories to $X$ coming from Kähler differentials:

\begin{itemize}
 \item \textit{Kähler-de Rham cohomology} (\cite{ArapuraKang}): This is the hypercohomology of the de Rham complex on $X$: 
 $$ H^\bullet_\KdR(X) := \mathbb H^\bullet(X, \Omega_X^\bullet).$$
 \item \textit{Algebraic de Rham cohomology} (\cite{Hartshorne-dR}): This is defined as the hypercohomology of the formal completion of the de Rham complex on $Y$ along $X$:
 $$ H^\bullet_\dR(X \hookrightarrow Y) := \mathbb H^\bullet(Y/_X, \Omega^\bullet_Y/_X).$$
\end{itemize}
If $\mathcal I_X$ denotes the ideal sheaf of $X$ in $Y$, the natural projection
 $$ \Omega^\bullet_Y/_X = {\underleftarrow{\lim}}_j\, \Omega_Y^\bullet/\mathcal I_X^j \to \Omega_Y^\bullet/\mathcal I_X \to \Omega^\bullet_X$$
 is compatible with the exterior derivative and is hence a morphism of complexes. Taking hypercohomology, this gives rise to a natural functorial comparison map $$H^\bullet_\dR(X \hookrightarrow Y) \to H^\bullet_\KdR(X).$$ 
 By \cite{Hartshorne-dR}*{Theorem~II.1.4}, the algebraic de Rham cohomology of $X$ does not depend on the embedding $X \hookrightarrow Y$, and we will hence simply write $H_\dR^\bullet(X)$. In particular, if $X$ is smooth, the above comparison map is an isomorphism.

In general, algebraic de Rham cohomology always gives the ``correct'' Betti numbers in the following sense:

\begin{theorem}[\cite{Hartshorne-dR}*{Theorem IV.1.1}]\label{T:singcomp}
 Suppose $K = \mathbb C$. Then there is a natural isomorphism
 $$H^\bullet_\dR(X) \cong H^\bullet_\sing(X^\an)$$
 between the algebraic de Rham cohomology of $X$ and the singular cohomology of the associated analytic space $X^\an$.
\end{theorem}
However, for singular $X$, Kähler-de Rham cohomology tends to give bigger Betti numbers:
\begin{theorem}[\cite{BloomHerrera}*{Corollary 3.14}]\label{T:KdRcomp}
 Suppose $K = \mathbb C$ and that $X$ is complete or has at most isolated singularities. Then $H^\bullet_\sing(X^\an, \mathbb C)$ is a direct summand of $H^\bullet_\KdR(X)$.
\end{theorem}

Note that by the Lefschetz principle, we can always find an embedding of $K$ into $\mathbb C$, and both cohomology theories are compatible with field extensions.

\subsection{Cohomological tools}\label{SS:tools}

For future reference, we briefly mention some standard facts and tools. The notation $H^i(X)$ without any subscript refers to either algebraic de Rham or Kähler-de Rham cohomology. The dimension of $H^i(X)$ as a $K$-vector space will be denoted by $h^i(X)$.

\begin{fact}[Betti numbers of affine and projective space]\label{F:AnPn}
 Let $n \geq 1$ be an integer. Then
$$ h^i(\mathbb A^n_K) = \begin{cases}
                         1 & \text{ if } i = 0,\\
			 0 & \text{ if } i \neq 0
                        \end{cases}
\quad \text{ and } \quad
h^i(\mathbb P^n_K) = \begin{cases}
                         1 & \text{ if } i \in \{0, 2, \dots, 2n\},\\
			 0 & \text{ otherwise.}
                        \end{cases}
$$
\end{fact}

\begin{fact}[Excision]\label{F:Excision}
 Let $Y$ be a $K$-variety and let $X \subseteq Y$ be a closed subscheme. Then there is a long exact sequence
 $$\dots \to H^i(Y) \to H^i(Y \setminus X) \to H^{i+1}_X(Y) \to H^{i+1}(Y) \to \dots,$$
 where $H^\bullet_X(Y)$ denotes local cohomology of $Y$ with support in $X$. Moreover, if $U \subseteq Y$ is an open subscheme containing $X$, then there is a natural isomorphism $H^i_X(Y) \cong H^i_X(U)$.
\end{fact}

\begin{fact}[Smooth Gysin sequence]\label{F:GysinSmooth}
Suppose that $K$ is algebraically closed. Let $Y$ be a smooth $K$-variety and let $X \subseteq Y$ be a closed smooth subvariety of codimension $r$. Then there is a long exact sequence
$$ \dots \to H^i(Y) \to H^i(Y \setminus X) \xrightarrow{\rho} H^{i+1-2r}(X) \to H^{i+1}(Y) \to  \dots$$
The map $\rho$ is called \textit{Poincaré residue map}.
\end{fact}

\begin{proof}
 This is the content of \cite{Hartshorne-Ample}*{Theorem~III.8.3}.
\end{proof}

For algebraic de Rham cohomology, there is a good theory $H^i_{c,\dR}$ of cohomology with compact support such that $H^i_{c,\dR}(X) = H^i_\dR(X)$ whenever $X$ is proper \cite{BCF}.

\begin{fact}[Compact support Gysin sequence]\label{F:GysinCompact}
 Let $Y$ be an arbitrary $K$-variety and let $X \subseteq Y$ be a closed subscheme. Then there is a long exact sequence
$$ \dots \to H^i(Y)_{c,\dR} \to H^i_{c,\dR}(X) \to H^{i+1}_{c,\dR}(Y \setminus X) \to H^{i+1}_{c,\dR}(Y) \to  \dots$$
\end{fact}

\begin{fact}[Poincaré duality]\label{F:Poincare}
 Let $Y$ be a smooth $K$-variety of dimension $n$ and let $X$ be a closed subscheme. Then there is a perfect pairing
 $$ H^\bullet_{c,\dR}(X) \times H^{2n-\bullet}_{\dR,X}(Y) \to K.$$
 In particular $h^i(Y) = h^{2n-i}_{c, \dR}(Y)$ and if $Y$ is proper, then $h^i(Y) = h^{2n-i}(Y)$ for all $i$.
\end{fact}

In the case of algebraic de Rham cohomology of smooth varieties, the smooth Gysin sequence arises as Poincaré dual of the compact support Gysin sequence.

\begin{fact}[Cohomological dimension]\label{F:dim}
 For any $K$-variety $X$, $H^i(X) = 0$ and $H^i_{c,\dR}(X) = 0$ for $i < 0$ and $i > 2 \dim X$.
\end{fact}

\begin{fact}[Cohomological dimension of affines]\label{F:affine}
 Let $X$ be affine of dimension $n$. Then $H^i_{\dR}(X) = 0$ for $i > n$. Moreover, if $X \subseteq \mathbb A^{n+1}_K$ is a hypersurface, then $H^i_{\KdR}(X) = 0$ for $i > n$.
\end{fact}
\begin{proof}
 The general result on algebraic de Rham cohomology follows from Theorem~\ref{T:singcomp} and the corresponding vanishing for Stein spaces \cite{BloomHerrera}*{Corollary 3.15}.
 
 Suppose now that $X$ is a hypersurface in affine $(n+1)$-space. Denote by $R$ the coordinate ring of $\mathbb A^{n+1}_K$ and suppose that the hypersurface $X$ is defined by $f \in R$. Consider the natural surjection
 \begin{align*}
  \Omega^\bullet_R \to \Omega^\bullet_R \tensor R/(f) \to \Omega^\bullet_{R/(f)}.
 \end{align*}
This is compatible with the exterior derivative $d$ and gives thus a short exact sequence
$$ 0 \to \mathcal K^\bullet \to  \Omega^\bullet_R \to \Omega^\bullet_{R/(f)} \to 0
$$
of complexes. This yields in turn a long exact sequence in cohomology, which reads
$$ \dots \to H^i( \mathcal K^\bullet) \to H^i(\mathbb A^{n+1}) \to H^i_\KdR(X) \to H^{i+1}( \mathcal K^\bullet) \to \dots $$
Since $\Omega^i_R = 0$ for $i > n+1$, we have $\mathcal K^i = 0$ and $\Omega^i_{R/(f)} = 0$ and thus $H^i_\KdR(X) = 0$ for $i > n+1$. Moreover, $H^n(\mathbb A^{n+1}) = 0$ and $\mathcal K^{n+2} = 0$ imply $H^{n+1}_\KdR(X) = 0$.  \qedhere
\end{proof}

\subsection{Algebraic de Rham cohomology of projective hypersurfaces}

Consider the projective space $\mathbb P^n_K$ over $K$. Let $\overline X \subseteq \mathbb P^n_K$ be a hypersurface. Then the first half of the algebraic de Rham cohomology of $\overline X$ is well understood:

\begin{lemma}[Lefschetz hyperplane theorem]\label{L:Lefschetz}
 The natural restriction $H^i_\dR(\mathbb P^n) \to H^i_\dR(\overline X)$ is an isomorphism for $i \leq n-2$ and injective for $i \leq n-1$.
\end{lemma}
\begin{proof}
The restriction map fits into the Gysin sequence with compact support:
$$ \dots \to H^i_{c,\dR}(\mathbb P^n \setminus \overline X) \to H^i_\dR(\mathbb P^n) \to H^i_\dR(\overline X) \to H^{i+1}_{c,\dR}(\mathbb P^n \setminus \overline X) \to \dots$$
By Poincaré duality, $$H^i_{c,\dR}(\mathbb P^n \setminus \overline X) \cong H^{2n-i}(\mathbb P^n \setminus \overline X)^\vee.$$
The variety $\mathbb P^n \setminus \overline X$ is smooth and affine of dimension $n$. Hence by Fact~\ref{F:affine} we conclude that $H^{2n-i}(\mathbb P^n \setminus \overline X)$ vanishes for $2n-i > n$, i.e. $i < n$.
\end{proof}

If $\overline X$ is smooth, then the Lefschetz hyperplane theorem combined with Poincaré duality on $\overline X$ gives almost all the Betti numbers:
\begin{corollary}\label{C:smoothcoh}
 Suppose that $\overline X$ is smooth. Then $h^i(\overline X) = h^i(\mathbb P^n)$ for $i \notin \{n-1, 2n\}$.
\end{corollary}
Due to dimension reasons, $h^{2n}(\overline X) = 0$. The middle Betti number of a smooth hypersurfaces can be computed by the methods of Griffiths \cite{Griffiths-PeriodsI}.

However, for singular hypersurfaces, Poincaré duality may fail. From now on, we will focus on the case of isolated singularities, i.e. the singular locus of $\overline X$ has dimension $0$.

\begin{lemma}\label{L:singcoh}
 Suppose that $\overline X$ has only isolated singularities. Then 
 \begin{itemize}
  \item $h^i_\dR(\overline X) = h^i(\mathbb P^n)$ for $i \notin \{n-1,n,2n\}$,
  \item $h^i_\KdR(\overline X) = h^i(\mathbb P^n)$ for $n+1 \leq i \leq 2n-1$.
 \end{itemize}
\end{lemma}
\begin{proof}
 Again by dimension reasons, $h^{2n}_\dR(\overline X) = 0$, whereas $h^{2n}(\mathbb P^n) = 1$. Denote by $\Sigma$ the singular locus of $\overline X$. By Bertini's theorem, after possibly extending the base field, there is a hyperplane $H \subseteq \mathbb P^n$ such that $\Sigma \cap H = \emptyset$ and $\overline Y := \overline X \cap H$ is a smooth hypersurface in $H \cong \mathbb P^{n-1}$. In particular
 $$h^i(\overline Y) = h^i(\mathbb P^{n-1}) = h^i(\mathbb P^n), \quad i \leq n-3.$$
 
 Let $X := \overline X \setminus \overline Y$. This is a singular hypersurface in $\mathbb A^n$, so  $H^i(X) = 0$ for $i \geq n$. Using the long exact sequence
 $$ \dots \to H^{i-1}(X) \to H^i_{\overline Y}(\overline X) \to H^i(\overline X) \to H^i(X) \to \dots,$$
 we obtain $$H^i_{\overline Y}(\overline X) \cong H^i(\overline X), \quad i \geq n+1.$$
 Since $\overline Y$ is a smooth closed subscheme of $\overline X \setminus \Sigma$, excision and Poincaré duality on $\overline X \setminus \Sigma$ yield
 $$h^i_{\overline Y,\dR}(\overline X) = h^i_{\overline Y,\dR}(\overline X \setminus \Sigma) = h^{2n-2-i}(\overline Y).$$
 Thus
 $$h^i_\dR(\overline X) = h^{2n-2-i}(\mathbb P^n) = h^i(\mathbb P^n), \quad i \geq n+1.$$
 For Kähler-de Rham cohomology, we have again $H^i_\KdR(\overline X) \cong H^i_{\overline Y, \KdR}(\overline X \setminus \Sigma)$ for $i \geq n+1$. Note that $H^i_{\overline Y}(\overline X \setminus \Sigma)$ fits into a long exact sequence
 $$ \dots \to H^{i-1}(\overline X \setminus \Sigma) \to H^{i-1}(X \setminus \Sigma) \to H^i_{\overline Y}(\overline X \setminus \Sigma) \to H^i(\overline X \setminus \Sigma) \to H^i(X \setminus \Sigma) \to \dots$$
 in both cohomology theories. In particular, since $\overline X \setminus \Sigma$ and $X \setminus \Sigma$ are smooth,
 \begin{align*}
h^i_\KdR(\overline X) &= h^i_{\overline Y, \KdR}(\overline X \setminus \Sigma) \\
&= \dim \ker \left(H^i(\overline X \setminus \Sigma) \to H^i(X \setminus \Sigma)\right) + \dim \coker \left(H^{i-1}(\overline X \setminus \Sigma) \to H^{i-1}(X \setminus \Sigma)\right) \\
&= h^i_{\overline Y, \dR}(\overline X \setminus \Sigma) \\
&=h^i_\dR(\overline X)  , \quad i \geq n+1.
 \end{align*}
 
 The case $i \leq n-2$ in algebraic de Rham cohomology is handled by the Lefschetz hyperplane theorem \ref{L:Lefschetz}.
\end{proof}

\subsection{Defect} 
 Let $\overline X \subseteq \mathbb P^n_K$ be a hypersurface with at most isolated singularities.
 \begin{itemize}
  \item The \textit{defect of $\overline X$ in algebraic de Rham cohomology} is $\delta_\dR(\overline X) := h^n_\dR(\overline X) - h^n_\dR(\mathbb P^n)$.
  \item  The \textit{defect of $\overline X$ in Kähler-de Rham cohomology} is $\delta_\KdR(\overline X) := h^n_\KdR(\overline X) - h^n_\KdR(\mathbb P^n)$.
 \end{itemize}
 Clearly $\delta_\dR(\overline X) \leq \delta_\KdR(\overline X)$ by Theorem~\ref{T:KdRcomp}. We will show in Corollary~\ref{C:defindep} that in fact equality holds. Hence we can simply speak of \textit{defect} and denote it by $\delta(\overline X)$. Furthermore, we will say that $\overline X$ \textit{has defect} if $\delta(\overline X) > 0$.

\begin{remarks} More remarks on defect:
 \begin{itemize}
  \item In other words, $\overline X$ has defect if the $n$-th Betti numbers of $\overline X$ and $\mathbb P^n$ do not agree. In particular, a hypersurface $\overline X$ with defect has to be singular.
  \item Since $\overline X$ is assumed to have at most isolated singularities, $h^i(\overline X) = h^i(\mathbb P^n)$ for all $i \neq \{n-1, n, 2n\}$ by Lemma~\ref{L:singcoh}.
  \item By Lemma~\ref{L:Com1}, the defect of $\overline X$ is always non-negative.
  \item Defect depends only on the Betti numbers. Since our cohomology theories involved are compatible with field extensions, we may as well assume that $K$ is algebraically closed.
 \end{itemize}
\end{remarks}

\subsection{Defect and cokernels}
In the remainder of this subsection, we will give some cohomological characterizations of defect for hypersurfaces with isolated singularities. Let $n \geq 3$ be an integer and let $\overline X \subseteq \mathbb P^n_K$ be a hypersurface with singular locus $\Sigma$. Assume that $\dim \Sigma = 0$. Again by Bertini's theorem, we can find a hyperplane $H \subseteq \mathbb P^n_K$ such that $H \cap \Sigma = \emptyset$ and $\overline X \cap H$ is smooth. Define $X := \overline X\setminus(\overline X \cap H) $; this is a singular hypersurface in $\mathbb P^n \setminus H \cong \mathbb A^n$.

\begin{lemma}\label{L:Com1} Consider the long exact sequence
$$  \dots \to  H^{n-1}(\overline X\setminus \Sigma) \xrightarrow{\alpha} H^n_\Sigma(\overline X) \to H^n(\overline X) \to H^n(\overline X \setminus \Sigma) \to H^{n+1}_\Sigma(\overline X) \to \dots$$
Then $ \delta(\overline X) = \dim \coker \alpha.$
\end{lemma}
\begin{proof} The proof consists of a few technical computations. We first assume that $n \geq 4$.
\begin{itemize} 
 \item $\overline X \setminus \Sigma$ is a smooth closed subvariety of codimension one in $\mathbb P^n \setminus \Sigma$. The corresponding Gysin sequence is
$$ \dots \to H^{n+1}(\mathbb P^n \setminus \overline X) \to H^n(\overline X \setminus \Sigma) \to H^{n+2}(\mathbb P^n \setminus \Sigma) \to H^{n+2}(\mathbb P^n \setminus \overline X) \to \dots$$
Since $\mathbb P^n \setminus \overline X$ is smooth and affine of dimension $n$, there is an isomorphism
$$ H^n(\overline X\setminus \Sigma) \cong H^{n+2}(\mathbb P^n \setminus \Sigma).$$
$\Sigma$ is a closed subvariety of codimension $n$ in $\mathbb P^n$. The associated Gysin sequence is
$$ \dots \to H^{n-2}_\dR(\Sigma) \to H^{n-2}(\mathbb P^n) \to H^{n-2}_{c,\dR}(\mathbb P^n \setminus \Sigma) \to H^{n-3}_\dR(\Sigma) \to \dots$$
Since $n \geq 4$, we can use $\dim \Sigma = 0$ to obtain
$$ h^{n+2}(\mathbb P^n \setminus \Sigma) = h^{n-2}_{c,\dR}(\mathbb P^n \setminus \Sigma) =  h^{n-2}(\mathbb P^n).$$
This shows that $h^n(\overline X \setminus \Sigma) = h^{n-2}(\mathbb P^n) = h^n(\mathbb P^n)$.
\item Since $\Sigma$ lies inside the affine part $X \subseteq \overline X$, $H^{n+1}_{\Sigma}(X) = H^{n+1}_{\Sigma}(\overline X)$. Using the two Gysin sequences for $X \setminus \Sigma \subseteq \mathbb A^n \setminus \Sigma$ and $\Sigma \subseteq \mathbb A^n$ gives
$$h^n(X \setminus \Sigma) = h^{n+2}(\mathbb A^n \setminus \Sigma) = h^{n-2}(\mathbb A^n) = 0.$$
On the other hand, $H^{n+1}(X) = 0$ since $X$ is an affine hypersurface of dimension $n-1$. The excision sequence for $\Sigma \subseteq X$ then yields $H^{n+1}_\Sigma(X) = 0$.
\item Putting this together,
$$ h^n(\overline X) =  \dim \coker \alpha + h^n(\overline X \setminus \Sigma) = \dim \coker \alpha + h^n(\mathbb P^n).$$
\item In the case $n = 3$, the long exact sequence in the statement of the lemma gives
$$ h^3(\overline X) = \dim \coker \alpha + h^3(\overline X \setminus \Sigma) - h^4_\Sigma(\overline X) + h^4(\overline X) - h^4(\overline X \setminus \Sigma) ,$$
where we used $H^5_\Sigma(\overline X) = 0$ for dimension reasons. Using Poincaré duality on $\overline X \setminus \Sigma$ and the compact support Gysin sequence
$$ 0 \to H^0_{c,\dR}(\overline X \setminus \Sigma) \to H^0_\dR(\overline X) \to H^0_{\dR}(\Sigma) \to H^1_{c,\dR}(\overline X \setminus \Sigma) \to H^1_{\dR}(X) = 0$$
we obtain
$$ h^3(\overline X \setminus \Sigma) - h^4(\overline X \setminus \Sigma) = h^1_{c,\dR}(\overline X \setminus \Sigma) - h^0_{c,\dR}(\overline X \setminus \Sigma) = h^0_\dR(\Sigma) - h^0_\dR(\overline X).$$
Note that $h^0_\dR(\overline X) = 1$ and $h^4(\overline X) = 1$ by Lemma~\ref{L:singcoh}. Thus
$$ h^3(\overline X) = \dim \coker \alpha + h^0_\dR(\Sigma)  - h^4_\Sigma(\overline X).$$
Moreover,
$$ h^4_\Sigma(\overline X) = h^4_\Sigma(X) = h^3(X \setminus \Sigma) = h^1_{c, \dR}(X \setminus \Sigma) = h^0_\dR(\Sigma),$$
the last step uses the compact support Gysin sequences
$$ \dots \to H^0_{c,\dR}(X) \to H^0_\dR(\Sigma) \to H^1_{c, \dR}(X \setminus \Sigma) \to H^1_{c, \dR}(X) \to \dots $$
and
$$ 0 = H^i_{c,\dR}(\mathbb A^n) \to H^i_{c,\dR}(X) \to H^{i+1}_{c,\dR}(\mathbb A^3 \setminus X) = 0, \quad i = 0, 1.$$
Consequently
\begin{align*}
 \delta(\overline X) = h^3(\overline X) = \dim \coker \alpha. & \qedhere
\end{align*}
\end{itemize}
\end{proof}

The open immersion $X \hookrightarrow \overline X$ induces a commutative ladder 
\begin{align*}
\begin{CD}
 \dots @>>> H^{n-1}(\overline X) @>>> H^{n-1}(\overline X\setminus \Sigma) @>\alpha>> H^n_\Sigma(\overline X) @>>> H^n(\overline X) @>>> \dots \\
@. @VVV @VVV @VVV @VVV\\
 \dots @>>> H^{n-1}(X) @>\vartheta >> H^{n-1}(X\setminus \Sigma) @>>> H^n_\Sigma(X) @>>> H^n(X) @>>> \dots 
\end{CD}
\end{align*}
of long exact sequences.

\begin{lemma}\label{L:Com2}
We have $\delta(\overline X) = \dim \coker \beta$, where 
$$\beta: H^{n-1}(\overline X \setminus \Sigma) \to H^{n-1}(X \setminus \Sigma)/\vartheta(H^{n-1}(X))$$
is the map induced by $X\setminus \Sigma \hookrightarrow \overline X \setminus \Sigma$.
\end{lemma}
\begin{proof}
 As $H^n(X) = 0$, the natural map $H^{n-1}(X \setminus \Sigma) \to H^n_\Sigma(X)$ is surjective. Since its kernel is given by the image of $\vartheta$,
$$H^{n-1}(X \setminus \Sigma)/\vartheta(H^{n-1}(X)) \to H^n_\Sigma(X)$$
is an isomorphism. The singular locus $\Sigma$ lies inside the affine part $X$, so the natural map $H^n_\Sigma(\overline X) \to H^n_\Sigma(X)$ is an isomorphism as well. Therefore $\coker \beta \cong \coker \alpha$, which finishes the proof by the preceding Lemma~\ref{L:Com1}.
\end{proof}

\begin{corollary}\label{C:defindep}
Let $\overline X \subseteq \mathbb P^n$ be a hypersurface with at most isolated singularities. Then $\delta_\dR(\overline X) = \delta_\KdR(\overline X)$. In particular $h^n_\dR(\overline X) = h^n_\KdR(\overline X)$.
\end{corollary}
\begin{proof}
 By Theorem~\ref{T:KdRcomp}, it remains to show the inequality $\delta_\dR(\overline X) \geq \delta_\KdR(\overline X)$. To this end, note that the comparison map between algebraic and Kähler-de Rham cohomology yields a commutative diagram
 \begin{align*}
  \begin{CD}
  H^{n-1}_\dR(X) @>\vartheta_\dR >> H^{n-1}_\dR(X \setminus \Sigma) \\
   @VVV @VV\simeq V\\
  H^{n-1}_\KdR(X) @>\vartheta_\KdR>> H^{n-1}_\KdR(X \setminus \Sigma).
  \end{CD}
 \end{align*}
This gives a surjection
$$ H_\dR^{n-1}(X \setminus \Sigma)/\vartheta_\dR(H^{n-1}_\dR(X)) \twoheadrightarrow  H^{n-1}_\KdR(X \setminus \Sigma)/\vartheta_\KdR(H^{n-1}_\KdR(X)).$$
If $\beta_\dR, \beta_\KdR$ denote the two versions of the map $\beta$ of Lemma~\ref{L:Com2}, then this gives rise to a surjection $ \coker \beta_\dR \twoheadrightarrow  \coker \beta_\KdR$. Hence
$$ \delta_\dR(\overline X) = \dim \coker \beta_\dR \geq \dim \coker \beta_\KdR = \delta_\KdR(\overline X).$$
The reverse inequality follows from Theorem~\ref{T:KdRcomp}. \qedhere
\end{proof}

\begin{corollary}\label{C:Cone}
 Suppose that $H^{n-1}(X) = 0$. Then $h^n(\overline X) = h^{n-2}(\overline X \setminus X)$.
\end{corollary}
\begin{proof}
 By Lemma~\ref{L:Com2}, the number $\delta(X)$ equals the dimension of the cokernel of the restriction map 
 $$ H^{n-1}(\overline X \setminus \Sigma) \to H^{n-1}(X \setminus \Sigma),$$
 which fits into a long exact sequence
 $$ \dots \to H^{n-1}(\overline X \setminus \Sigma) \to H^{n-1}(X \setminus \Sigma) \to H^n_{\overline X \setminus X}(\overline X \setminus \Sigma) \to H^n(\overline X \setminus \Sigma) \to \dots$$
By Poincaré duality, $h^n_{\overline X \setminus X}(\overline X \setminus \Sigma) = h^{n-2}(\overline X \setminus X)$. By the proof Lemma~\ref{L:Com1}, $h^n(\overline X \setminus \Sigma) = h^n(\mathbb P^n)$ and $h^n(X \setminus \Sigma) = 0$. Therefore
\begin{align*}
h^n(\overline X) = \delta(\overline X) + h^n(\mathbb P^n) = h^n_{\overline X \setminus X}(\overline X \setminus \Sigma) - h^n(\overline X \setminus \Sigma) + h^n(\mathbb P^n) = h^{n-2}(\overline X \setminus X). & \qedhere
\end{align*}
\end{proof}
\begin{remark}
This gives several examples of hypersurfaces with defect: In particular, if $\overline X$ is the cone over a smooth projective hypersurface $Y \subseteq \mathbb P^{n-1}$, then $\delta(\overline X) = h^{n-2}(Y) - h^n(\mathbb P^n)$. For example, any cone over a nonsingular plane curve of positive genus has defect.
\end{remark}

We finish this section with another cohomological characterization of defect: Using the smooth Gysin sequences for $\overline X \setminus \Sigma \subseteq \mathbb P^n \setminus \Sigma$ and $X \setminus \Sigma \subseteq \mathbb A^n \setminus \Sigma$ respectively, we get a commutative diagram
\begin{align*}
\begin{CD}
\dots @>>> H^n(\mathbb P^n \setminus \overline X) @>>> H^{n-1}(\overline X \setminus \Sigma)@>>> \dots\\
@. @VVV @VVV\\
\dots @>>> H^n(\mathbb A^n \setminus X) @>\rho >> H^{n-1}(X \setminus \Sigma)@>>> \dots,
\end{CD}
\end{align*}
where $\rho$ is the Poincaré residue.

\begin{lemma}\label{L:Com3}
We have $\delta(\overline X) \leq \dim \coker \gamma$, where
$$\gamma: H^n(\mathbb P^n \setminus \overline X) \to H^n(\mathbb A^n \setminus X)/\rho^{-1}(\vartheta(H^{n-1}(X)))$$
is the map induced by the open immersion $\mathbb A^n \setminus X \hookrightarrow \mathbb P^n \setminus \overline X$. Moreover, equality holds if $n$ is even.
\end{lemma}
\begin{proof} 
One checks that $H^n(\mathbb A^n \setminus \Sigma) = H^{n+1}(\mathbb A^n \setminus \Sigma) = 0$, so $\rho$ is an isomorphism. We obtain a commutative diagram
\begin{align*}
\begin{CD}
 H^n(\mathbb P^n \setminus \overline X) @> \sigma >> H^{n-1}(\overline X \setminus \Sigma)\\
@V\gamma VV @VV \beta V\\
 H^n(\mathbb A^n \setminus X) / \rho^{-1}(\vartheta(H^{n-1}(X))) @>\simeq > \rho> H^{n-1}(X \setminus \Sigma) / \vartheta(H^{n-1}(X)),
\end{CD}
\end{align*}
where $\beta$ is as in Lemma~\ref{L:Com2}. Thus
$$ \dim \coker \gamma = \dim \coker (\beta \circ \sigma) \geq \dim \coker \beta = \delta(\overline X).$$
If $n$ is even, then the map $\sigma$ is surjective, since the preceding term $H^{n+1}(\mathbb P^n \setminus \Sigma)$ in the Gysin sequence vanishes. Hence in this case, $\coker(\beta \circ \sigma) = \coker \beta$. \qedhere
\end{proof}

\subsection{Differential forms on hypersurface complements}

We keep the notations from the previous subsection. Suppose that the hypersurface $\overline X$ is defined by the homogeneous polynomial $F \in K[x_0,\dots,x_n]_d$. Moreover, assume that the hyperplane $H$ is given by the vanishing of $x_0$. Let $f = F(1,x_1,\dots,x_n) \in K[x_1,\dots,x_n]$ denote the defining polynomial of $X$ in $\mathbb A^n = \mathbb P^n \setminus \{x_0 = 0\}$.

In view of Lemma~\ref{L:Com3}, the defect of $\overline X$ may be described by investigating the top-dimensional cohomology of the hypersurface complements $\mathbb P^n \setminus \overline X$ and $\mathbb A^n \setminus X$. Fortunately, these spaces can be explicitly described. Both varieties in question are smooth and affine of dimension $n$, so their $n$-th algebraic de Rham cohomology is just a quotient of the module of $n$-forms on their coordinate rings. More precisely:
\begin{lemma}\label{L:Explicit}\quad
\begin{enumerate}[{\normalfont(1)}] 
 \item  $H^n(\mathbb P^n \setminus \overline X)$ is generated by
 $$\left\{ \frac{G\Omega}{F^k} \middle| \, G \in K[x_0,\dots,x_n]_{kd-n-1}, k \geq 0 \right\},$$
 where
 $$ \Omega := \sum_{i=0}^n (-1)^i x_i \,dx_1 \wedge \dots \wedge \widehat{dx_i} \wedge \dots \wedge dx_n.$$
\item  $H^n(\mathbb A^n \setminus X)$ is generated by
$$\left\{ \frac{g\omega}{f^k} \middle| \, g \in K[x_1,\dots,x_n], k \geq 0 \right\},$$
where $\omega := dx_1 \wedge \dots \wedge dx_n$.
\item The natural restriction map is given by
$$H^n(\mathbb P^n \setminus \overline X) \to H^n(\mathbb A^n \setminus X), \quad \left[\frac{G\Omega}{F^k}\right] \to  \left[\frac{g\omega}{f^k} \right],$$
where $g$ is the dehomogenization of $G$.
\end{enumerate}
\end{lemma}
\begin{proof}
 (2) and (3) are immediate. For (1), see e.g. \cite{Dimca}*{Chapter 6}.
\end{proof}

\subsection{Pole-order filtration}
In the notation of Lemma~\ref{L:Com3}, let $V := \rho^{-1}(\vartheta(H^{n-1}(X)))$. For $k \geq 0$, define the pole-order filtration $P^k$ on $H^n(\mathbb P^n\setminus \overline X)$ resp.\ $H^n(\mathbb A^n\setminus X)/V$ as the image of differential forms of the type $G\Omega/F^k$ resp.\ $g\omega/f^k$. Since $G\Omega/F^k = FG\Omega/F^{k+1}$ and similarly in the affine case, these are ascending filtrations. Note that these are slightly different to the ones given in Dimca's article \cite{Dimca-Betti}.

The pole-order filtration gives rise to the $k$-th graded objects
\begin{align*}
  \Gr_P^k H^n(\mathbb P^n\setminus \overline X) &:= P^kH^n(\mathbb P^n\setminus \overline X)/P^{k-1}H^n(\mathbb P^n\setminus \overline X),\\
   \Gr_P^k H^n(\mathbb A^n\setminus X) &:= P^kH^n(\mathbb A^n\setminus X)/P^{k-1}H^n(\mathbb A^n\setminus X), \quad k\geq 0,
\end{align*}
with the convention that $P^{-1} = \{0\}$. The natural restriction
$$\gamma: H^n(\mathbb P^n \setminus \overline X) \to H^n(\mathbb A^n \setminus X)/V$$
induces maps  $\Gr_P^k(\gamma)$ on the corresponding graded objects. In view of Lemma~\ref{L:Com3}, there is an immediate corollary:

\begin{corollary}\label{C:Graded}
If $\overline X$ has defect, then there is an integer $k \geq 0$ such that $\Gr^k_P(\gamma)$
is not surjective.
\end{corollary}

Set $S := K[x_0,\dots,x_n]$ and $R := K[x_1,\dots,x_n]$. The explicit description of the cohomology groups given in Lemma~\ref{L:Explicit} yields a commutative diagram
\begin{align*}
 \begin{CD}
  S_{kd-n-1} @>>> R \\
@VVV @VV\varphi V\\
\Gr^k_P H^n(\mathbb P^n \setminus X) @>\Gr^k_P(\gamma)>> \Gr^k_P (H^n(\mathbb A^n \setminus X)/V)
 \end{CD}
\end{align*}
for any $k \geq 0$ with surjective vertical arrows and the horizontal arrows being the natural restriction maps. We can actually make the top right corner smaller:

\begin{lemma}\label{L:Jacobian}
Let $\varphi: R \to \Gr^k_P (H^n(\mathbb A^n \setminus X)/V)$ be as in the above diagram. Let $J(f)$ denote the ideal in $R$ spanned by the partial derivatives of $f$. Then:
\begin{enumerate}[{\normalfont(1)}]
 \item For $k \geq 2$, the map $\varphi$ factors through $R/((f) + J(f))$.
 \item For $k = 1$, the map $\varphi$ factors through $R/((f)+J(f)^3)$.
 \item $\Gr^0_P H^n(\mathbb P^n \setminus \overline X) = \Gr^0_P (H^n(\mathbb A^n \setminus X)/V) = 0$. 
\end{enumerate}
\end{lemma}
\begin{proof} 
If $g \in (f)+ J(f)$, then there are polynomials $h_0, \dots, h_n$ such that $$g = h_0 f + \sum_{i=1}^n h_i \frac{\partial f}{\partial x_i}.$$
The class of $h_0 f \omega/f^k$ vanishes in the graded object $\Gr_P^k$ by definition. One computes that
$$d\left( (-1)^i \frac{h_i}{f^{k-1}} dx_1 \wedge \dots \wedge \widehat{dx_i} \wedge \dots \wedge dx_n \right) = (k-1) \cdot h_i \cdot \frac{\partial f}{\partial x_i} \cdot \frac{\omega}{f^k } - \frac{\partial h_i}{\partial x_i} \cdot \frac{\omega}{f^{k-1}}.$$
Hence if $k \geq 2$, we can rewrite the cohomology class of $h_i \frac{\partial f}{\partial x_i} \omega/f^k$ as the class of a differential form with lower pole order. But such classes vanish in the graded object $Gr^k_P$ by definition of the pole-order filtration.

For $k = 0$ observe at first that $\Gr^0_P H^n(\mathbb P^n \setminus \overline X) $ is generated by $S_{n-1} = 0$. If $h \in R$ is any polynomial, then the above relation for $k = 1$ shows that all forms of the type $\frac{\partial h}{\partial x_i} \omega$ vanish in $\Gr^0_P (H^n(\mathbb A^n \setminus \overline X)/V) $. But any form can be written in this way.

If $k = 1$, we cannot apply the pole-order reduction trick anymore. However, we can use the space $V$: Let $\eta \in \Omega^{n-1}_R$ be a global $(n-1)$-form. Then the class of $\eta$ in $\Omega^{n-1}_{R/(f)}$ lies in the kernel of $d:\Omega^{n-1}_{R/(f)} \to \Omega^{n}_{R/(f)}$ if and only if $d \eta = f \xi + \zeta \wedge df$ for some $\xi \in \Omega^n_R, \zeta \in \Omega^{n-1}_R$. Such an $\eta$ defines a cohomology class in $H^{n-1}(X)$.

In the notation of Lemma~\ref{L:Com2}, restricting to the open $X \setminus \Sigma$ via $\vartheta$ and applying the inverse of the Poincaré residue map $\rho$ (see \cite{Hartshorne-Ample}*{ Theorem~III.8.3}), we get a map
$$\rho^{-1} \circ \vartheta: W := \{\eta \in \Omega^{n-1}_R \mid \exists\, \xi \in \Omega^n_R, \zeta \in \Omega^{n-1}_R: d \eta = f \xi + \zeta \wedge df \} \to V, \quad \eta \mapsto \left[\frac{\eta \wedge df}{f}\right].$$
In particular, all forms inside the image of this map will vanish in $\Gr^1_P (H^n(\mathbb A^n \setminus X)/V)$.

We will now give a description in terms of polynomials: Write  $$\eta := \sum_{i=1}^n (-1)^i h_i \cdot  dx_1 \wedge \dots \wedge \widehat{dx_i} \wedge \dots \wedge dx_n, \quad h_i \in R.$$ Then
$$ \eta \wedge df = \sum_{i=1}^n h_i \frac{\partial f}{\partial x_i} \omega
 \quad \text{ and } \quad
 d\eta = \sum_{i=1}^n \frac{\partial h_i}{\partial x_i}.$$
Now let $g \in (f) + J(f)^3$. Then
$$ g = fh' + \sum_{i,j,k=1}^n h_{ijk} \frac{\partial f}{\partial x_i} \frac{\partial f}{\partial x_j} \frac{\partial f}{\partial x_k} = fh'+ \sum_{i=1}^n \left(\sum_{j,k=1}^n h_{ijk} \frac{\partial f}{\partial x_j} \frac{\partial f}{\partial x_k}\right)  \frac{\partial f}{\partial x_i}, \quad h', h_{ijk} \in R$$
and $$\sum_{i=1}^n \frac{\partial}{\partial x_i} \left(\sum_{j,k=1}^n h_{ijk} \frac{\partial f}{\partial x_j} \frac{\partial f}{\partial x_k}\right)\in J(f).$$
Thus if $(g-fh')\omega = \eta \wedge df$ as above, then $d\eta = \zeta \wedge df$ for some $\zeta$ and hence $\eta \in W$. In particular, inside $H^n(\mathbb A^n \setminus X)$,
$$ \left[\frac{g \omega}{f} \right] = \left[\frac{(g-fh') \omega}{f} \right] + \left[\frac{f h' \omega}{f} \right] = \left[\frac{\eta \wedge df}{f} \right] + [h' \omega] = [\rho^{-1}(\vartheta(\eta))] + 0 \in V.$$
Consequently, the map $\varphi$ factors through $R/((f)+J(f)^3)$.
\end{proof}

\subsection{Defect and Tjurina number}

We are now in shape to prove the main theorem of this section.
\begin{theorem}\label{T:milnor}
 Let $\tau := \dim_K K[x_1,\dots,x_n]/((f) + J(f))$ be the global Tjurina number of $\overline X$. If $\overline X$ has defect, then
   $$\tau \geq \frac{d-n+1}{n^2+n+1}.$$
   Moreover, if the map $\Gr_P^k(\gamma)$ is not surjective for some $k \geq 2$, then $$\tau \geq kd-n+1.$$
\end{theorem}
\begin{proof}
 Since $\overline X$ has defect, there is an integer $k \geq 0$ such that $\Gr^k_P(\gamma)$ is not surjective by Corollary~\ref{C:Graded}. Using Lemma~\ref{L:Jacobian}~(3), we can assume that $k \geq 1$.
 
 Assume first that $k \geq 2$. Then the non-surjectivity of some $\Gr^k_P(\gamma)$ implies the non-surjectivity of the natural restriction map
$ S_{kd-n-1} \to T(f)$, where $T(f) := R/((f) + J(f))$ denotes the global Tjurina algebra of $f$. Since $\dim \Sigma = 0$, $T(f)$ is a finite-dimensional $K$-algebra. Applying Poonen's trick \cite{Poonen-Bertini}*{Lemma~2.1(b)} shows that the image of $S_i$ in $T(f)$ stricty increases with $i$ until it fills the whole space. In particular, the restriction map has to be surjective for $i \geq \tau - 1$. From this, one infers that $kd - n - 1 \leq \tau - 2$, whence $\tau \geq kd - n + 1$.

If $k = 1$, then the same argument shows that $\dim_K R/((f) + J(f)^3) \geq d - n + 1$. Using the exact sequences of $T(f)$-modules
$$ 0 \to J(f)^i/J(f)^{i+1} \to R/((f) + J(f)^{i+1}) \to R/((f)+J(f)^i) \to 0 $$
for $i = 1, 2$, we obtain
$$ \dim_K R/((f) + J(f)^3) = \dim_K J(f)^2/J(f)^3 + \dim_K J(f)/J(f)^2 + \dim_K T(f).$$
Since $J(f)^i/J(f)^{i+1}$ can be generated by $n^i$ elements, it has length at most $n^i$ as $T(f)$-module. Thus
\begin{align*}
  d - n + 1 \leq \dim_K R/((f) + J(f)^3) \leq n^2 \tau + n \tau + \tau = (n^2+n+1)\cdot \tau. & \qedhere
\end{align*}
\end{proof}

\subsection{Local computations}

In the case that the singularities of $\overline X$ are weighted homogeneous, we can use the methods of Dimca \cite{Dimca-Milnor} to improve the bound of Theorem~\ref{T:milnor}:

\begin{lemma} In the notations of Lemma~\ref{L:Jacobian}, suppose that $\overline X$ has only weighted homogeneous singularities. Then the natural map
$$ \varphi: R \to \Gr_P^1(H^n(\mathbb A^n \setminus X)/V)$$
factors through $R/((f) + J(f))$.
\end{lemma}
\begin{proof}
 By the Lefschetz principle, assume that $K \subseteq \mathbb C$ and use analytic de Rham cohomology.  As in \cite{Dimca-Milnor}*{Section 3}, the map $\Gr^1_P(\gamma)$ can be described as the natural restriction
 $$\Gr^1_P(\gamma): \Gr_P^1 H^n(\mathbb P^n \setminus X) \to \bigoplus_{x \in \Sigma} \Gr_{P_x}^1 H^n(\Omega^\bullet_{f,x}),$$
 where $\Omega^\bullet_{f,x}$ denotes the localization of the holomorphic de Rham complex $\Omega_{\mathbb C^n, x}^\bullet$ with respect to $f$, and $P_x$ is the corresponding local pole-order filtration. In particular, for any $x \in \Sigma$ there is a natural surjection
 $$\varphi_x: \mathcal O_{\mathbb C^n, x} \to \Gr_{P_x}^1 H^n(\Omega^\bullet_{f,x}), \quad g \mapsto \left[ \frac{g}{f} \, dx_1 \wedge \dots \wedge dx_n \right].$$
 
 Suppose now that the singularity of $\overline X$ at $x$ is contact-equivalent to a weighted homogeneous singularity. Then there is a biholomorphic coordinate change $\psi$ sending $x$ to $(0, \dots, 0)$ such that $f' = \psi(f)$ is a weighted homogeneous polynomial. Moreover, $\psi$ induces an isomorphism of the local Tjurina algebras of $f$ at $x$ and $f'$ at $0$, respectively. 
 
 Take a polynomial $h \in (f) + J(f)$. Under the natural map 
 $$R  \to H^n(\Omega_{f,x}^\bullet) \xrightarrow{\simeq} H^n(\Omega_{f',0}^\bullet)$$
 induced by $\psi$, the class $[h/f dx_1 \wedge \dots \wedge dx_n]$ is sent to some $[h'/f' dx_1' \wedge \dots \wedge dx_n']$ with $h'$ lying in the analytic ideal $(f') + J(f') \subseteq \mathcal O_{\mathbb C^n, 0}$. However, the calculation \cite{Dimca-Milnor}*{Example 3.6} shows that $[h'/f' dx_1' \wedge \dots \wedge dx_n'] = 0$.
\end{proof}

Applying the same methods as in the proof of Theorem~\ref{T:milnor}, this yields:
\begin{corollary}\label{C:whom}
 Suppose that $\overline X$ has at most weighted homogeneous singularities. If $\overline X$ has defect, then
 $\tau \geq d - n + 1.$
\end{corollary}

A well-known application is the following:
\begin{corollary}[\cite{Dimca-Betti}*{Proposition 3.4}]
 Suppose that $\overline X$ has at most ordinary double points as singularities. If $\overline X$ has defect, then  $ \tau \geq \frac{dn}{2}-n+1.$
\end{corollary}
\begin{proof}
 The polynomial $f' = x_1^2 + \dots + x_n^2$ is weighted homogeneous of degree $2$ with respect to the weights $(1, \dots, 1)$. Since the local cohomology piece $\Gr^k_{P_0} H^n(\Omega_{f',0}^\bullet)$ is spanned by homogeneous forms of degree $2k - n $, it vanishes for $k \neq \frac{n}{2}$. In particular, $\overline X$ has no defect if $n$ is odd, as the map $\Gr^k_P(\gamma)$ is always surjective. For even $n$, defect implies that $\Gr^k_P(\gamma)$ is surjective for $k \neq \frac{n}{2}$ and not surjective for $ k = \frac{n}{2}$. This concludes the proof by the second part of Theorem~\ref{T:milnor}.
\end{proof}
\begin{remarks} Let $\overline X \subseteq \mathbb P^n$ be a nodal hypersurface.
\begin{itemize}
 \item  One can actually show that if $\overline X$ has defect and $\dim \overline X = 3$, then $\tau \geq (d-1)^2$, see \cite{Cheltsov} or \cite{Kloosterman-MaximalFamilies}*{Theorem 4.1}. The latter proof carries over to higher dimensions. 
 \item For even $n$, it is conjectured in \cite{Kloosterman-MaximalFamilies} that $\tau \geq (d-1)^{n/2}$.
\end{itemize}
\end{remarks}

\subsection{Proof of Theorem~\ref{T:Main0}}

\begin{proof}
 The statements for algebraic de Rham and Kähler-de Rham cohomology follow from Theorem~\ref{T:milnor} and Corollary~\ref{C:whom}. Embedding $K$ into $\mathbb C$ using the Lefschetz principle, defect in algebraic de Rham is equivalent to defect in singular cohomology by Theorem~\ref{T:singcomp}. The étale version is a consequence of Artin's comparison theorem between étale and singular cohomology \cite{Artin-Etale}*{Theorem~5.2}.
\end{proof}

\subsection{Positive characteristic}\label{SS:charp}

If $K$ is a field of characteristic $p > 0$, both algebraic and Kähler-de Rham cohomology behave pathologically. For example, affine space has infinite Betti numbers. To remedy this, one needs a different kind of cohomology theory. Of course $\ell$-adic étale cohomology, where $\ell \neq p$ is a prime, is a reasonable choice, but it is hard to describe explicitly.

A different possibility is to choose rigid cohomology, which is a $p$-adic cohomology theory built in analogy to algebraic de Rham cohomology (see e.g.\ \cite{Berthelot-Finitude}, \cite{LeStum}). For hypersurface complements in affine or projective space, there is a similar description as in Lemma~\ref{L:Explicit}, replacing polynomials by overconvergent power series. However, the rigid cohomology of singular varieties is a rather mysterious object. To our knowledge, it is not even known whether $H^n(X) = 0$ holds for a singular affine hypersurface $X \subseteq \mathbb A^n$.

The field $K$ admits a ring of Witt vectors $W(K)$, denote its field of quotients by $Q(K)$. Let $F \in W(K)[x_0,\dots,x_n]_d$ be a homogeneous polynomial of degree $d$ with coefficients in the ring $W(K)$. Then $F$ defines a $W(K)$-scheme $\mathcal X$. Its generic fiber is the hypersurface $\mathcal X_\eta := \{F = 0\} \subseteq \mathbb P^n_{Q(K)}$. The special fiber $\mathcal X_s$ is a hypersurface in $\mathbb P^n_K$ defined by reducing $F$ modulo $p$. Both the rigid cohomology of $\mathcal X_s$ and the algebraic de Rham cohomology of $\mathcal X_\eta$ take values in $Q(K)$, and there is a natural cospecialization map relating them. This map is an isomorphism when $\mathcal X$ is smooth. For singular $\mathcal X$, this is no longer true: A simple example is given by $F = x_0^2 + x_1^2 + x_2^2 + x_3^2 + p x_4^2 \in \mathbb Z_p[x_0,\dots,x_4]$. The corresponding generic fiber $\mathcal X_\eta$ is a smooth hypersurface in $\mathbb P^4_{\mathbb Q_p}$ and hence $h^4_\dR(\mathcal X_\eta) = 1$. On the other hand, $\mathcal X_s \subseteq \mathbb P^4_
{\mathbb F_q}$ is not factorial, so $h^4_\rig(\mathcal X_s) > 1$ by Theorem~\ref{T:factorial}. If instead we choose $F = x_0^2 + x_1^2 + x_2^2 + x_3^2$, the special fiber does not change, but the generic fiber has defect as well.

This motivates the following question:

\begin{question}
Let $X  \subseteq \mathbb P^n_K$ be a hypersurface with defect in rigid cohomology. Does $X$ admit a lift $\mathcal X \subseteq \mathbb P^n_{W(K)}$ such that the generic fiber $\mathcal X_\eta \subseteq \mathbb P^n_{Q(K)}$ has defect in algebraic de Rham cohomology? 
\end{question}

If this question had an affirmative answer, then we could use the results of Section~\ref{S:char0}:
\begin{corollary}
 Let $X \subseteq \mathbb P^n_K$ be a hypersurface of degree $d$ with global Tjurina number $\tau$ admitting a lift with defect. Then
 $$ \tau \geq \frac{d-n+1}{n^2+n+1}.$$
\end{corollary}
\begin{proof}
 The Nakayama lemma implies that the Tjurina number cannot decrease after reduction mod $p$. Apply Theorem~\ref{T:milnor}.
\end{proof}

However, this question seems to be very delicate. By \cite{Vakil-Murphy}*{Theorem~1.1}, there are surfaces $S \subseteq \mathbb P^4$ that do not lift to characteristic zero. Such surfaces cannot be complete intersections, so no hypersurface $X$ containing $S$ can be factorial. In particular, if such an $X$ is defined over $\overline{\mathbb F_p}$, then $X$ will have defect by Theorem~\ref{T:factorial}. On the other hand, it is well possible that every lift of $X$ is factorial, as we cannot lift $S$.

\section{Resolution of singularities}\label{S:res}

\subsection{Cohomological preliminaries}

In this section, we relate defect of hypersurfaces to the number of singularities following the ideas presented in \cite{PRS}. Let $K$ be an algebraically closed field of characteristic $p \neq 2$. Denote by $H^\bullet$ one of these theories:
\begin{itemize}
 \item étale cohomology with coefficients in $\mathbb Q_\ell$, where $\ell \neq p$ is a prime,
 \item algebraic de Rham cohomology with coefficients in $K$ (if $p = 0$),
 \item rigid cohomology with coefficients in the field of quotients of the ring of Witt vectors of $K$ (if $p > 0$).
\end{itemize}
All these theories feature the cohomological facts \ref{F:AnPn}-\ref{F:affine} and the Lefschetz hyperplane theorem \ref{L:Lefschetz} with the small exception that it is not known whether $H_\rig^i(Z) = 0$ for singular affine varieties $Z$ and $i > \dim Z$. However, this will only be used for affine hypersurfaces with weighted homogeneous singularities, where the required statements follow from \cite{David-Thesis}*{§3.2}. The advantage is now that we can freely the cohomological proofs of Section~\ref{S:char0}.

Moreover, we will need two more cohomological tools.

\begin{lemma}[Long exact sequence of a proper birational morphism]\label{L:properbirational}
 Let $X$ be a complete variety over $K$. Further let $\pi: Y \to X$ be a proper birational morphism such that its restriction $\pi|_{Y \setminus E}: Y\setminus E \to X \setminus \Sigma$ is an isomorphism for certain closed subschemes $E \subseteq Y$ and $\Sigma \subseteq X$. Then there is a long exact sequence
 $$ \dots \to H^i(X) \to H^i(Y) \oplus H^i(\Sigma) \to H^i(E) \to H^{i+1}(X) \to \dots$$
\end{lemma}
\begin{proof}
See also \cite{Hartshorne-dR}*{Theorem~II.4.4} for algebraic de Rham cohomology and \cite{Kloosterman-Average}*{Proposition~2.3} for étale cohomology. Since cohomology with compact support is contravariant with respect to proper morphisms, the resolution $\pi$ induces a commutative ladder
 \begin{align*}
  \begin{CD}
   \dots @>>> H^i_c(X \setminus \Sigma) @>>> H^i_c(X) @>\beta>> H^i_c(\Sigma) @>>> H^{i+1}_c(X \setminus \Sigma) @>>> \dots\\
   @. @| @V\alpha VV @V\gamma VV @| @.\\
   \dots @>>> H^i_c(Y \setminus E) @>>> H^i_c(Y) @>\delta>> H^i_c(E) @>>> H^{i+1}_c(Y \setminus E) @>>> \dots
  \end{CD}
 \end{align*}
By diagram chasing, this yields a long exact sequence 
$$ \dots \to H^i_c(X) \xrightarrow{(\alpha, \beta)} H^i_c(Y) \oplus H^i_c(\Sigma) \xrightarrow{\gamma - \delta} H^i_c(E) \to H^{i+1}_c(X) \to \dots$$
Since $X, \Sigma, Y, E$ are all complete, we can omit the compact support.
\end{proof}

\begin{lemma}[Mayer-Vietoris sequence]\label{L:MVss}
 Let $X_1, \dots, X_r$ be projective varieties over $K$ and let $X := X_1 \cup \dots \cup X_r$. Suppose that the triple intersections $X_j \cap X_k \cap X_\ell$ are empty for pairwise distinct $j, k, \ell$. Then there is a long exact sequence
 $$ \dots \to H^i(X) \to \bigoplus_{j=1}^r H^i(X_j) \to \bigoplus_{1\leq j <k \leq r} H^i(X_j \cap X_k) \to H^{i+1}(X) \to \dots.$$
\end{lemma}
\begin{proof}
 In the algebraic de Rham case, let $X \hookrightarrow Y$ be a closed embedding into a smooth projective variety $Y$. Then there is a short exact sequence of formally completed de Rham complexes
 $$ 0 \to \Omega^\bullet_Y/_X \to \bigoplus_{j=1}^r \Omega^\bullet_Y/_{X_j} \to \bigoplus_{1 \leq j < k \leq r} \Omega^\bullet_Y/_{X_j \cap X_k} \to 0,$$
 compare \cite{Hartshorne-dR}*{Proposition~II.4.1}. It remains to apply hypercohomology. The proof for rigid cohomology is analogous: Embed $X$ into the closed fiber of a smooth formal scheme $\mathscr P$ and use the short exact sequence
 $$ 0 \to \Omega^\bullet_{]X[_{\mathscr P}} \to \bigoplus_{j=1}^r \Omega^\bullet_{]X_j[_{\mathscr P}} \to \bigoplus_{1 \leq j < k \leq r} \Omega^\bullet_{]X_j \cap X_k[_{\mathscr P}} \to 0.$$
 For étale cohomology, let $\iota_j$ resp. $\iota_{j,k}$ denote the inclusion of $X_j$ resp. $X_j \cap X_k$ into $X$ and take the long exact cohomology sequence of
 \begin{align*}
  0 \to \mathbb Q_\ell \to \bigoplus_{j=1}^r {\iota_j}_*\mathbb Q_\ell \to \bigoplus_{1 \leq j < k \leq r} {\iota_{j,k}}_* \mathbb Q_\ell \to 0. & \qedhere
 \end{align*}
\end{proof}

\subsection{Hypersurfaces with ordinary multiple points and $A_k$ singularities}

For a positive integer $n \geq 3$, let $X \subseteq \mathbb P^n_K$ be an irreducible hypersurface of degree $d$ with isolated singularities. Again, we define the \textit{defect} of $X$ as $$\delta(X) := h^n(X) - h^n(\mathbb P^n).$$ 

Suppose further that the singular points belong to the following classes:
\begin{itemize}
 \item \textit{Ordinary multiple points.} A point $x$ is an \textit{ordinary multiple point of multiplicity $m$} if the projectivized tangent cone at $x$ is the cone over a smooth degree $m$ hypersurface in $\mathbb P^{n-1}$ for some $m \geq 2$.
 \item \textit{$A_k$ singularities.} These are points whose completed local ring is isomorphic to $$K[[x_1,\dots,x_n]]/(x_1^{k+1}+x_2^2 + \dots +x_n^2)$$
 for some $k \geq 1$.
\end{itemize}
Note that an ordinary double point is an $A_1$ singularity, and this is the only common member of both families.

Let $\Sigma_O$ be the set of ordinary multiple points in $X$ of multiplicity $\geq 3$, and denote by $m_x$ the multiplicity of a point $x \in \Sigma_O$. Similarly, define $\Sigma_A$ to be the union of all $A_k$ points in $X$ for $k \geq 1$, and for an $A_k$ singularity $x \in \Sigma_A$ let $r_x := \lceil k/2 \rceil$.

The advantage of restricting to these two classes of singularities is the very explicit nature of a resolution of singularities:

\begin{proposition}\label{P:ResSing}
Let $X$ be as above. Then there is an embedded resolution of singularities $\pi: (Y \subseteq P) \to (X \subseteq \mathbb P^n)$ such that $P$ is a smooth $n$-fold obtained from $\mathbb P^n$ by a finite sequence of blowups in points. More precisely:
\begin{enumerate}[\normalfont (1)]
 \item $P$ is obtained by $\mathbb P^n$ as a sequence of
 $$s := \#\Sigma_O + \sum_{x \in \Sigma_A} r_x$$
 blowups in points.
 \item As a divisor on $P$, the strict transform $Y$ of $X$ is linearly equivalent to
 $$ d H - \sum_{x \in \Sigma_O} m_x \mathcal D_x - \sum_{x \in \Sigma_A} \sum_{i=1}^{r_x} 2i \cdot \mathcal E_{x,i} $$
 where
 \begin{itemize}
  \item $H$ is the pullback of a hyperplane,
  \item $\mathcal D_x \cong \mathbb P^{n-1}$ and $D_x := Y \cap \mathcal D_x$ is a smooth degree $m_x$ hypersurface in $\mathbb P^{n-1}$,
  \item $\mathcal E_{x,i}$ is obtained from $\mathbb P^{n-1}$ by $r_x-i$ blowups in points and $E_{x,i} := Y \cap \mathcal E_{x,i}$ is isomorphic to the blowup at the vertex of the cone over a smooth quadric in $\mathbb P^{n-2}$ for $i= 1, \dots, r_x-1$.
  \item $\mathcal E_{x,r_x} \cong \mathbb P^{n-1}$ and $E_{x, r_x} := Y \cap \mathcal E_{x, r_x}$ is isomorphic to a smooth quadric in $\mathbb P^{n-1}$ if $k$ is odd,
  \item $\mathcal E_{x,r_x} \cong \mathbb P^{n-1}$ and $E_{x, r_x} := Y \cap \mathcal E_{x, r_x}$ is isomorphic to the cone over a smooth quadric in $\mathbb P^{n-2}$ if $k$ is even.
  \item $E_{x,i} \cap E_{x,j} = \emptyset$ unless $|i - j| \leq 1$ and $E_{x,i} \cap E_{x,i+1}$ is isomorphic to a smooth quadric in $\mathbb P^{n-2}$ for $i = 1, \dots, r_x-1$.
  \item $E_{x,i} \cap E_{x,j} \cap E_{x,k} = \emptyset$ for pairwise distinct $i,j,k$.
 \end{itemize}
\end{enumerate}
\end{proposition}
\begin{proof}
 See \cite{PRS} for the case of ordinary multiple points and \cite{DaisRoczen}, \cite{Schepers} for details on resolving $A_k$ singularities.
\end{proof}

\subsection{Vanishing of local cohomology}

Before computing Betti numbers of the resolution, we remark that if $n$ happens to be odd, $A_k$ singularities do not contribute to defect:

\begin{lemma}\label{L:AkOdd}
 If $n$ is odd, then $H^n_{\Sigma_A}(X) = 0$. In particular, if $X$ has at most $A_k$ singularities, then $H^n(X) = 0$.
\end{lemma}
\begin{proof}
 The ``in particular'' statement follows from Lemma~\ref{L:Com1}.
 
 The result is well-known for de Rham cohomology in characteristic zero, see \cite{Dimca-Betti}*{Examples~1.9}. In general, the space $H^n_{\Sigma_A}(X)$ decomposes into the direct sum $\bigoplus_{x \in \Sigma_A} H^n_{\{x\}}(X)$. Moreover, $H^n_{\{x\}}(X)$ depends only on $(X, x)$ up to contact equivalence, see \cite{David-Local}*{Subsection~1.2} for étale and rigid cohomology.
 
 Thus we are left with computing $H^n_{\{0\}}(Z)$ for the variety $Z = \{x_1^k + x_2^2 + \dots + x_n^2 = 0\} \subseteq \mathbb A^n$. Consider the exact sequence
 $$ \dots \to H^{n-1}(Z) \to H^{n-1}(Z \setminus \{0\}) \to H^n_{\{0\}}(Z) \to H^n(Z) \to \dots$$
 Since $H^n(Z) = 0$, we only need to show that $H^{n-1}(Z \setminus \{0\}) = 0$. By Poincaré duality, $H^{n-1}(Z \setminus \{0\}) \cong H^{n-1}_c(Z \setminus \{0\})^\vee.$ Now let $\overline Z \subseteq \mathbb P^n$ denote the projective closure of $Z$. Then there is a compact support Gysin sequence
 $$  \dots \to H^{n-2}(\overline Z \setminus Z) \to H^{n-1}_c(Z \setminus \{0\}) \to H^{n-1}_c(\overline Z \setminus \{0\}) \to H^{n-1}(\overline Z \setminus Z ) \to H^n(Z \setminus \{0\}) \to \dots$$
 
 As in the proof of Lemma~\ref{L:Com1}, $H^n(Z \setminus \{0\}) = 0$. The variety $\overline Z \setminus Z$ is either a smooth quadric in $\mathbb P^{n-1}$ ($k = 1$) or a hyperplane of multiplicity $k$ ($k \geq 2$). In both cases, we have $h^{n-2}(\overline Z \setminus Z) = 0$ and $h^{n-1}(\overline Z \setminus Z) = 1$. Thus it suffices to show that $h^{n-1}_c(\overline Z \setminus \{0\}) = 1$.
 
 To this end, observe that by Proposition~\ref{P:ResSing}, $\overline Z$ has a resolution of singularities that lifts to characteristic zero. Applying proper and smooth base change (étale cohomology, \cite{Milne-EC}) or the Baldassarri-Chiarellotto comparison theorem (rigid cohomology, \cite{BaldassarriChiarellotto}*{Corollary~2.6}), we can reduce to the known de Rham cohomology case.
\end{proof}

\begin{remark}  Ordinary multiple points of multiplicity $\geq 3$ can cause defect on even-dimensional hypersurfaces: Let $X \subseteq \mathbb P^3$ be the projective cone over a smooth plane curve $C$ of degree $m \geq 3$. Then $h^3(X) = h^1(C) = (m-1)(m-2) > 0$ by Corollary~\ref{C:Cone}, so $X$ has defect.
\end{remark}

%

\subsection{Defect and Betti numbers of the resolution}

We will now give a cohomological criterion for defect using the embedded resolution of singularities $\pi$ from Proposition~\ref{P:ResSing}. First, we need the Betti numbers of $P$, which is obtained by $s$ successive blowups.
\begin{lemma}\label{L:BlowupCoh} We have
 $$ h^i(P) = \begin{cases}
                         s + 1 & \text{ if } i \in \{2, \dots, 2n-2\},\\
                         1 & \text{ if } i \in \{0, 2n\},\\
			 0 & \text{ otherwise.}
                        \end{cases} $$
\end{lemma}
\begin{proof}
Let $P_0 := \mathbb P^n$ and for $j = 1, \dots, s$ denote by $P_j$ the blowup of $P_{j-1}$ in a point.
 By Lemma~\ref{L:properbirational}, there is an exact sequence
 $$ \dots \to H^i(P_j) \to H^i(P_{j+1}) \oplus H^i(\{\text{pt}\}) \to H^i(\mathbb P^{n-1}) \to H^{i+1}(P_j) \to \dots$$
 Using the Betti numbers of projective space, the claim follows by induction.
\end{proof}

The next step is to compute some Betti numbers of the exceptional divisor $E$ associated to the resolution $\pi|_Y: Y \to X$, i.e.,
$$E := Y \cap  \left(\sum_{x \in \Sigma_O} D_x + \sum_{x \in \Sigma_A} \sum_{i=1}^{r_x} E_{x,i}\right).$$

\begin{lemma}\label{L:ExcEven}
 Suppose that $n$ is even. Then $h^{n-1}(E) = 0$ and $h^n(E) = s$.
\end{lemma}
\begin{proof}
 $E$ is the disjoint union of the divisors $D_x$, $x \in \Sigma_O$, and $E_x = \sum_{i=1}^{r_x} E_{x,i} $, $x \in \Sigma_A$. Hence we can treat each singularity type separately.
 \begin{itemize}
  \item $D_x$ for $x \in \Sigma_O$.  By the description given in Proposition~\ref{P:ResSing}, $D_x$ is isomorphic to a smooth degree $m_x$ hypersurface in $\mathbb P^{n-1}$. Hence by Corollary~\ref{C:smoothcoh}, $ h^i(D_x) = h^i(\mathbb P^{n-1})$ for $i \notin \{n-2, 2n-2\}$. In particular $h^{n-1}(D_x) = 0$ and $h^n(D_x) = 1$.
  \item $E_x$ for $x \in \Sigma_A$.  Let $Q$ be a smooth quadric in $\mathbb P^{n-2}$, let $C$ be the cone over $Q$ in $\mathbb P^{n-1}$ and denote by $B$ the blowup of $C$ in its vertex. Further let $S$ be a smooth quadric in $\mathbb P^{n-1}$. Using Corollary~\ref{C:smoothcoh} and Lemma~\ref{L:singcoh}, one computes that
  $$ h^i(Q) = h^i(C) = h^i(B) = h^i(S) = 0 \quad \text{ for all odd } i \geq n-1.$$
  
  Since there are no triple intersections between the components of $E_x$, Lemma~\ref{L:MVss} yields a long exact Mayer-Vietoris sequence
   $$ \dots \to H^q(E_x) \to \bigoplus_i H^{q}(E_{x,i}) \xrightarrow{d_q}  \bigoplus_{i < j}  H^{q}(E_{x,i} \cap E_{x,j}) \to H^{q+1}(E_x) \to \dots$$ 
  
  We claim that the maps $d_{n-2}$ and $d_n$ are surjective. Assuming this, we immediately have $h^{n-1}(E_x) = 0$ by the description given in Proposition~\ref{P:ResSing}.
  
  In the case $n = 4$, the $E_{x,i}$ are irreducible surfaces, so $h^4(E_x) = \sum_{i=1}^{r_x} h^4(E_{x,i}) = r_x$. For $n \geq 6$, one computes $h^n(Q) = h^n(S) = 1$ by Corollary~\ref{C:smoothcoh}, $h^n(C) = 1$ by Lemma~\ref{L:singcoh} and thus $h^n(B) = 2$. Therefore
  $$h^n(E_x) = (r_x-1) \cdot h^n(E_{x,i}) + h^n(E_{x,r_x}) - (r_x-1) \cdot h^n(E_{x,i} \cap E_{x,j}) = r_x.$$
  
  It remains to prove the surjectivity of 
  $$\bigoplus_i H^{q}(E_{x,i}) \to \bigoplus_{i < j} H^{q}(E_{x,i} \cap E_{x,j}) $$
  for $q = n- 2, n$. Since $E_{x,i} \cap E_{x,j}$ is empty unless $|i - j| = 1$,  this would follow from the surjectivity of all the maps
  $$H^{q}(E_{x,i}) \to H^{q}(E_{x,i} \cap E_{x,{i+1}}), \quad i = 1, \dots, r_x-1. $$
  But the intersection $E_{x,i} \cap E_{x,{i+1}} \cong Q$ is a smooth quadric inside the exceptional divisor $F \cong \mathbb P^{n-2}$ of the blowup of $C$ at its vertex. Thus the restricton morphism $H^q(F) \to H^q(Q)$ is surjective for $q = n-2, n$. Moreover, Lemma~\ref{L:properbirational} yields that there is an exact sequence
  $$ \dots \to H^q(C) \to H^q(B) \to H^q(F) \to H^{q+1}(C) \to \dots$$
  Using $h^{q+1}(C) = 0$, we obtain that the map $H^q(B) \to H^q(F)$ is surjective and so is the composition
  $$H^q(E_{x,i}) \xrightarrow{\simeq} H^q(B) \to H^q(F) \to H^q(Q) \xrightarrow{\simeq} H^q(E_{x,i} \cap E_{x,{i+1}}).$$
  \end{itemize}
  Summing up,
  \begin{align*}
    h^{n-1}(E) = 0 \quad\text{ and } h^n(E) = \sum_{x \in \Sigma_O} 1 + \sum_{x \in \Sigma_A} r_x = s. & \qedhere
  \end{align*}
\end{proof}

\begin{lemma}\label{L:ExcOdd}
 Suppose that $n$ is odd. Then $h^{n}(E) = 0$ and $h^n(X) \leq h^n(Y)$.
\end{lemma}
\begin{proof}
 The proof that $h^n(E) = 0$ is analogous to the proof of $h^{n-1}(E) = 0$ given in Lemma~\ref{L:ExcEven}.
 
 It remains to show the inequality $h^n(X) \leq h^n(Y)$. We first blow up the ordinary multiple points successively. This gives a partial resolution $\psi: Y_O \to X$ with $Y_O$ having at most $A_k$ singularities.  The morphism $\psi$ comes from an embedded resolution $P_O \to \mathbb P^n$, and thus we have the following commutative diagram by Lemma~\ref{L:properbirational}:
  \begin{align*}
   \begin{CD}
     H^{n-1}(P_O) @>>> \bigoplus_{x \in \Sigma_O} H^{n-1}(\mathcal D_x)  @>>> H^n(\mathbb P^n) \\
     @VVV @VVV @VVV @. @.\\
     H^{n-1}(Y_O) @>>> \bigoplus_{x \in \Sigma_O} H^{n-1}( D_x) @>>> H^n(X) @>>> H^n(Y_O) @>>> H^n(E),
   \end{CD}
  \end{align*}
Since $\mathcal D_x \cong \mathbb P^{n-1}$, and $D_x$ is a smooth hypersurface therein, the natural restriction map $H^{n-1}(\mathcal D_x) \to H^{n-1}(D_x)$ is an isomorphism. Together with $H^n(\mathbb P^n) = 0$ this implies that $H^{n-1}(Y_O) \to  \bigoplus_{x \in \Sigma_O} H^{n-1}( D_x)$ is surjective and thus $H^n(X) \cong H^n(Y_O)$.

Let $D_O := \sum_{x \in \Sigma_O} D_x$. Then
$$H^n_{\psi^{-1}(\Sigma_A)}(Y_O) \cong H^n_{\psi^{-1}(\Sigma_A)}(Y_O \setminus D_O) \cong H^n_{\Sigma_A}(X \setminus \Sigma_O) = H^n_{\Sigma_A}(X) = 0$$
by Lemma~\ref{L:AkOdd}. Resolving $Y_O$, we obtain our smooth hypersurface $Y$ in $P$ with the exceptional divisor $E_A \cong \sum_{x \in \Sigma_A} \sum_{i=1}^{r_x} E_{x,i}$. This resolution gives a commutative diagram
\begin{align*}
 \begin{CD}
  H^n_{\psi^{-1}(\Sigma_A)}(Y_O) @>>> H^n(Y_O) @>>> H^n(Y_O \setminus \psi^{-1}(\Sigma_A)) \\
  @. @VVV @|\\
  @. H^n(Y) @>>> H^n(Y \setminus E_A)
 \end{CD}
\end{align*}
It follows that $H^n(Y_O) \to H^n(Y_O \setminus \psi^{-1}(\Sigma_A)) \to H^n(Y \setminus E_A)$ is injective. This implies that the map $H^n(Y_O) \to H^n(Y)$ is injective as well.

Consequently, $h^n(X) = h^n(Y_O) \leq h^n(Y)$.
\end{proof}

With the two lemmas above, we obtain a simple formula for the defect of $X$:
\begin{proposition}\label{P:Exc} The defect of $X$ may be computed as follows:
 $$ \delta(X) = \begin{cases}
                 h^n(Y) - s - 1 & \text{ if } n \text{ is even,}\\
                 h^n(Y) & \text{ if } n \text{ is odd.}
                \end{cases}$$
\end{proposition}
\begin{proof}
 Applying Lemma~\ref{L:properbirational} to $\pi|_Y:Y \to X$, there is a long exact sequence
 $$ \dots \to H^{n-1}(Y) \to H^{n-1}(E) \to H^n(X) \to H^n(Y) \to H^n(E) \to H^{n+1}(X) \to \dots $$
 Suppose first that $n$ is even. Using Lemma~\ref{L:ExcEven}, $h^{n-1}(E) = 0$ and $h^n(E) = s$. By Lemma~\ref{L:singcoh} we have $h^{n+1}(X) = 0$. It follows that $h^n(Y) = h^n(X) + s$. If $n$ is odd, then inserting $h^n(E) = 0$ into the above long exact sequence implies $h^n(X) \geq h^n(Y)$. On the other hand, $h^n(X) \leq h^n(Y)$ by Lemma~\ref{L:ExcOdd}, so that $h^n(X) = h^n(Y)$.
\end{proof}

\subsection{Defect and ampleness of the strict transform}

We keep the notation of the previous subsection. If the strict transform $Y$ of $X$ happens to be an ample divisor in $P$, then - embedding $P$ into a suitable projective space - the Lefschetz hyperplane theorem~\ref{L:Lefschetz} shows that the restriction map
$$ H^{n-2}(P) \to H^{n-2}(Y)$$
is an isomorphism. Applying Poincaré duality on $Y$, $h^{n-2}(P) = h^n(Y)$. Hence we have the following corollary of Proposition~\ref{P:Exc} and Lemma~\ref{L:BlowupCoh}:

\begin{corollary}\label{C:Ample}
 Suppose that $Y$ is ample in $P$. Then $\delta(X) = 0$.
\end{corollary}

Finally, we can relate ampleness of $Y$ to the number of singularities of $X$.
\begin{lemma}\label{L:Ample}
 Suppose that $$\sum_{x \in \Sigma_O} m_x + \sum_{x \in \Sigma_A} 2 r_x < d.$$
 Then $Y$ is ample in $P$.
\end{lemma}
\begin{proof}
 This is a variant of \cite{PRS}*{Theorem~4.1}. By Proposition~\ref{P:ResSing}, inside $\Pic(P)$,
 \begin{align*}
  Y &= d H - \sum_{x \in \Sigma_O}  m_x D_x - \sum_{x \in \Sigma_A}\sum_{i=1}^{r_x} 2i \cdot \mathcal E_{x,i} \\
  &= d H - \sum_{x \in \Sigma_O} m_x D_x - 2 \sum_{x \in \Sigma_A} \sum_{i=1}^{r_x}  \underset{=:\, \widetilde {\mathcal E_{x,i}}}{\underbrace{ \sum_{j=i}^{r_x} \mathcal E_{x,j} }}\\
  &= \left(d - \sum_{x \in \Sigma_O} m_x - \sum_{x \in \Sigma_A} 2 r_x \right) H + \sum_{x \in \Sigma_O} m_x (H-D_x) + 2 \sum_{x \in \Sigma_A} \sum_{i=1}^{r_x}  (H - \widetilde {\mathcal  E_{x,i}}).
 \end{align*}
 
 Since $H$ is the pullback of a hyperplane, the linear system $|H|$ has no base points. Using the hypothesis,
 $$\left|\left(d - \sum_{x \in \Sigma_O} m_x - \sum_{x \in \Sigma_A} 2 r_x \right) H\right|$$
 is base-point free as well.  If $x \in X$ is a singular point, then $x$ is scheme-theoretically cut out by hyperplanes. In particular, its ideal sheaf twisted by $\mathcal O(1)$ is globally generated, and so are the pullbacks $\mathcal O_P(H-D_x)$ and $\mathcal O_P \left (H- \widetilde {\mathcal E_{x,1}}\right)$, respectively. Similarly, $\mathcal O_P \left(H- \widetilde {\mathcal E_{x,i}}\right)$ is globally generated for any $i$. In total, $\mathcal O_P(Y)$ is a globally generated invertible sheaf on $P$.
 
 It follows that if $C \subseteq P$ is an irreducible curve, then $Y.C \geq 0$. In order to show that $Y$ is ample, it suffices to show that such an intersection $Y.C$ is always positive. If $\pi_*C$ is a curve on $\mathbb P^n$, then by the projection formula
 $$ H.C = (\pi_* C).\mathcal O(1) > 0,$$
 thus $Y.C > 0$.
 
 If $C$ is contracted by $\pi$, then $H.C = 0$ again by the projection formula. By base-point freeness of $|H - D_x|$ and $|H - \widetilde {\mathcal E_{x,i}}|$, all the intersection numbers $D_x.C$ and $\widetilde {\mathcal E_{x,i}}.C$ are hence nonpositive. The Picard group of $P$ is spanned by $H$, the $\mathcal D_x$ and the $\widetilde {\mathcal E_{x,i}}$. Since $P$ is projective, there must be integers $h, d_x, e_{x,i}$ such that the divisor
 $$ A := h H + \sum_{x \in \Sigma_O} d_x \mathcal D_x + \sum_{x \in \Sigma_A} e_{x,i} \widetilde {\mathcal E_{x,i}}$$
 is ample and thus $A.C > 0$. In particular, at least one of the intersection products $D_x.C$ or $\widetilde {\mathcal E_{x,i}}.C$ is nonzero and hence strictly negative. This implies $Y.C > 0$.
 
 Consequently, $Y$ is ample in $P$. 
\end{proof}

\begin{remark}
 This proof does not carry over to singular points of type $D_k$ or $E_k$. For $n = 4$, the standard embedded resolution of these singularities has the property that the $s$ exceptional divisors of the resolution $P \to \mathbb P^n$ break into several components when intersecting with the strict transform $Y$ of $X$. In particular, $h^4(E) \geq s + 1 = h^2(P)$. But then by Lemma~\ref{L:properbirational} 
 $$ h^4(Y) \geq h^4(X) + h^4(E) \geq h^4(X) + h^2(P) \geq 1 + h^2(P),$$
 thus $h^4(Y) = h^2(Y) \neq h^2(P)$. Consequently, $Y$ cannot be ample in $P$ in virtue of the Lefschetz hyperplane theorem.
 
  However, in case that the ground field is of characteristic zero, the Hodge numbers of resolutions of hypersurfaces with at most $ADE$ singularities were investigated by Rams \cite{Rams}*{§4}.
\end{remark}

\subsection{Proof of Theorem~\ref{T:MainRes}}

\begin{proof}
 Suppose that $X$ has defect. Let $\pi: (Y \subseteq P) \to (X \subseteq \mathbb P^n)$ be the embedded resolution from Proposition~\ref{P:ResSing}. By Corollary~\ref{C:Ample}, $Y$ cannot be ample in $P$. Now Lemma~\ref{L:Ample} implies that
 \begin{align*}
  \sum_{x \in \Sigma_O} m_x + \sum_{x \in \Sigma_A} 2 r_x \geq d. & \qedhere
 \end{align*}
\end{proof}

\section{Factorial threefold hypersurfaces over $\overline{\mathbb F_p}$}\label{S:factorial}

Let $K$ be a field and $X \subseteq \mathbb P^4_K$ be a hypersurface defined by a homogeneous polynomial $f \in K[x_0,\dots,x_4]$.

$X$ is \textit{factorial} if the homogeneous coordinate ring $K[x_0,\dots,x_4]/(f)$ is a unique factorization domain. By \cite{Hartshorne}*{Exercise II.6.3}, $X$ is factorial if and only if the natural map  $\Pic(X) \to \Cl(X)$ is an isomorphism, i.e., if and only if every Weil divisor on $X$ is linearly equivalent to a Cartier divisor.

Furthermore, $X$ is called $\mathbb Q$-\textit{factorial} if the map $\Pic(X) \to \Cl(X)$ becomes an isomorphism after tensoring with $\mathbb Q$, i.e. if every Weil divisor on $X$ is linearly equivalent to a $\mathbb Q$-Cartier divisor.

\begin{theorem}\label{T:factorial}
 Suppose $K \subseteq \overline{\mathbb F_p}$. Let $X \subseteq \mathbb P^4_K$ be a hypersurface with at most isolated singularities. If $h^4_\et(X) = 1$ or $h^4_\rig(X) = 1$, then $X$ is factorial.
\end{theorem}
\begin{remark}
 The corresponding statement in characteristic zero is shown in \cite{PRS}*{Proposition~3.2}.
\end{remark}

From now on, let $X \subseteq \mathbb P^4_K$ be a hypersurface defined over $K = \overline {\mathbb F_p}$ with zero-dimensional singular locus $\Sigma$. Since $X$ is a threefold, \cite{CossartPiltant-ResII} provides a resolution of singularities $\pi: Y \to X$. Denote by $E$ the exceptional divisor and by $s$ the number of its irreducible components.

\begin{lemma}\label{L:ResPic}
With the above notations,
$$ \rk \Cl(X) = \rk \Pic(Y) - s.$$
\begin{proof}
 Since $\Sigma$ has codimension at least two in $X$,
 $$\Cl(X) \cong \Cl(X \setminus \Sigma) \cong \Cl(Y \setminus E).$$
Let $E_1, \dots, E_s$ denote the irreducible components of the exceptional divisor $E$. Then there is a standard exact sequence
$$ \bigoplus_{i=1}^s \mathbb Z \cdot E_i \to \Cl(Y) \to \Cl(Y \setminus E) \to 0.$$
This sequence is also exact on the left: Suppose $\sum_{i=1}^s a_i [E_i] = 0 \in \Cl(Y)$ for $a_1, \dots, a_s \in \mathbb Z$. If $H \subseteq Y$ is general hyperplane, then $D := \sum_{i=1}^s a_i (E_i \cap H)$ is linearly equivalent to $0$ as a divisor on the surface $Y \cap H$. However, as in \cite{Fulton-IT}*{Example 2.4.4}, $D$ has negative self-intersection, contradicting that $[D] = 0 \in \Cl(Y \cap H)$.
\end{proof}
\end{lemma}

\begin{lemma}\label{L:ResRig}
For both étale and rigid cohomology,
$$ h^4(Y) - s \leq h^4(X).$$
\end{lemma}
\begin{proof}
Since $H^4(\Sigma) = 0$ as $\dim \Sigma = 0$, this follows from the long exact sequence
$$\dots \to H^4(X) \to H^4(Y) \oplus H^4(\Sigma) \to H^4(E) \to \dots$$
Lemma~\ref{L:properbirational}.
\end{proof}

In order to compare Picard rank and Betti numbers, we need the following result on the étale cycle class map:

\begin{lemma}\label{L:cycle}
 Let $Z$ be a smooth projective variety over $K$. Then the étale cycle class map 
$$ \Pic(Z) \tensor \mathbb Q_\ell \to H^2(Z, \mathbb Q_\ell(1))$$
is injective.
\end{lemma}
\begin{proof}
Let $\ell$ be a prime not equal to $\characteristic(K)$. The étale cycle class map tensored with $\mathbb Q_\ell$ factors as
$$ \Pic(Z) \tensor \mathbb Q_\ell \xrightarrow{\alpha} \NS(Z) \tensor \mathbb Q_\ell \xrightarrow{\beta} H^2(Z, \mathbb Q_\ell(1)),$$
where $\NS(Z)$ denotes the Néron-Severi group of $Z$. As in \cite{Milne-EC}*{pp. 216--217}, one obtains that $\beta$ is is injective. The kernel of $\alpha$ is precisely $\Pic^0(Z) \tensor \mathbb Q_\ell$. But since $K = \overline{\mathbb F_p}$, the group $\Pic^0(Z)$ is torsion \cite{Keel-Bpf}*{Lemma 2.16}. Hence $\alpha$ is injective as well.
\end{proof}

\begin{corollary}\label{C:h4Qfac}
For both étale and rigid cohomology, we have $\rk \Cl(X) \leq h^4(X)$. In particular, if $h^4(X) = 1$, then $X$ is $\mathbb Q$-factorial.
\end{corollary}
\begin{proof} By Lemma~\ref{L:cycle}, $ \rk \Pic(Y) \leq h^2_\et(Y, \mathbb Q_\ell(1)) = h^2_\et(Y, \mathbb Q_\ell).$
As étale and rigid cohomology are both Weil cohomologies and $Y$ is defined over some finite field, applying the comparison theorem of Katz-Messing \cite{KatzMessing}*{Corollary 1} yields $h^2_\et(Y) = h^2_\rig(Y).$
Now Poincaré duality on $Y$ gives $h^2(Y) = h^4(Y)$. Thus, with the help of Lemma~\ref{L:ResPic} and Lemma~\ref{L:ResRig},
\begin{align*}
 \rk \Cl(X) = \rk \Pic(Y) - s \leq h^4(Y) - s \leq h^4(X). & \qedhere
\end{align*}
\end{proof}

Finally, we need to proceed from $\mathbb Q$-factoriality to factoriality.

\begin{lemma}\label{L:facQfac}
 If $X$ is $\mathbb Q$-factorial, then $X$ is factorial.
\end{lemma}
\begin{proof}
We follow the proof of \cite{PRS}*{Proposition~2.15}. Since $X$ is normal and Cohen-Macaulay, the proof of \cite{Hartshorne-Gorenstein}*{Proposition~2.15} generalizes and gives an exact sequence
$$ 0 \to \Pic(X) \to \Cl(X) \to \bigoplus_{x \in \Sigma} \Cl(\mathcal O_{X,x}).$$
In particular, there is an injection
$$ \Cl(X)/\Pic(X) \hookrightarrow \bigoplus_{x \in \Sigma} \Cl(\mathcal O_{X,x}).$$
By hypothesis, $\Cl(X)/\Pic(X)$ is a torsion group. Fix $x \in \Sigma$. By \cite{DLM}*{Corollary 2.10}, the Picard group of the punctured spectrum $U_x$ of $\mathcal O_{X,x}$ is torsion-free. Since $X$ has only isolated singularities, $\Pic(U_x) \cong \Cl(\mathcal O_{X,x})$, see \cite{Fossum}*{Proposition 18.10}. Consequently, $\Cl(X)/\Pic(X)$ is a torsion subgroup of a torsion-free group and hence trivial. Thus $X$ is factorial.
\end{proof}

\begin{proof}[Proof of Theorem~\ref{T:factorial}]
Since étale and rigid cohomology behave well with respect to base change, $X \times_{\Spec K} \Spec \overline {\mathbb F_p}$ is factorial by Corollary~\ref{C:h4Qfac} and Lemma~\ref{L:facQfac}. In other words, if $S$ denotes the homogeneous coordinate ring of $X$, then $S \tensor_K \overline{\mathbb F_p}$ is factorial. But this implies that $S$ and hence $X$ are factorial.
\end{proof}

\section{Density of hypersurfaces without defect}\label{S:density}

Let $K = \mathbb F_q$ be a finite field of characteristic $\neq 2$. By a result of Poonen, the asymptotic density of smooth hypersurfaces in $\mathbb P^n$ defined over $K$ is computed as follows:
\begin{theorem}[Poonen's Bertini theorem, \cite{Poonen-Bertini}*{Theorem~1.1}]\label{T:Poonen}
 \begin{align*}
  \lim_{d \to \infty} \frac{\#\{f \in K[x_0,\dots,x_n]_d \mid \{f = 0\} \subseteq \mathbb P^n_{K} \text{ is smooth}\}}{\#K[x_0,\dots,x_n]_d} = \frac{1}{\zeta_{\mathbb P^n_K}(n+1)}.
 \end{align*}
\end{theorem}

Here, $\zeta_{\mathbb P^n_K}$ denotes the Hasse-Weil zeta function of $\mathbb P^n_K$, which is simply given by
$$ \zeta_{\mathbb P^n_K}(s) = \prod_{i=1}^n (1-q^{i-s}) \quad \text{ for } s \in \mathbb C,\quad \operatorname{Re}(s) > n.$$

One trivial remark is that the limit in Theorem~\ref{T:Poonen} is smaller than $1$, so that a ``random'' hypersurface is smooth with a probability strictly less than 100\%. However, it is true that hypersurfaces with few singularities compared to the degree form a set of density $1$:

\begin{theorem}[\cite{Lindner-Density}*{Corollary~5.9}]\label{T:DensityMilnor}
Fix a constant $c > 0$. Then
 \begin{align*}
  \lim_{d \to \infty} \frac{\#\{f \in K[x_0,\dots,x_n]_d \mid \tau(f) \leq c \cdot d \}}{\#K[x_0,\dots,x_n]_d} =  1,
 \end{align*}
 where $\tau(f)$ denotes the global Tjurina number of the hypersurface $\{f = 0\} \subseteq \mathbb P^n_K$.
\end{theorem}

If Theorem~\ref{T:Main0} held over finite fields, then this would imply that hypersurfaces without defect form a set of density $1$. However, so far, we can only use the restricted singularity types from Theorem~\ref{T:MainRes}.

\begin{lemma}\label{L:AkDensity}
 Let $x \in \mathbb A^n_K$ be a closed point with residue field $\kappa(x)$. Fix a positive integer $d$ and choose a polynomial $f \in K[x_1,\dots,x_n]_{\leq d}$ uniformly at random. Then the probability that $\{f = 0\}$ has at most an $A_k$ singularity for some $k \geq 1$ in $x$ is at least
 $$ 1 - \#\kappa(x)^{-n-3}.$$
\end{lemma}
\begin{proof}
Let $\mathcal O_x$ be the local ring of $\mathbb A^n_K$ at $x$ and denote by $\mathfrak m_x$ its maximal ideal. Let
$$[f] = f_0 + f_1 + f_2 \in \mathcal O_x/\mathfrak m_x^3 \quad \text{ with } \deg f_i = i, \quad i = 0, 1, 2,$$
be the $2$-jet of $f$ at $x$. Define $X$ to be the hypersurface $\{f = 0\} \subseteq \mathbb A^n_K$. Then:
\begin{enumerate}[(1)]
 \item $X$ does not pass through $x$ $\Leftrightarrow$ $f_0 \neq 0$,
 \item $X$ is smooth at $x$ $\Leftrightarrow$ $f_0 = 0$ and $f_1 \neq 0$,
 \item $X$ is has an ordinary double point at $x$ $\Leftrightarrow$ $f_0 = 0$, $f_1 = 0$ and $f_2$ is a quadratic form of rank $n$,
 \item $X$ is has an $A_k$ singularity for some $k \geq 2$ at $x$ $\Leftrightarrow$ $f_0 = 0$, $f_1 = 0$ and $f_2$ is a quadratic form of rank $n-1$.
\end{enumerate}
The vector space $\mathcal O_x/\mathfrak m_x^3$ has dimension $1 + n + \frac{n(n+1)}{2}$ over $\kappa(x)$. Let $r := \#\kappa(x)$. The probability that $X$ has at most an $A_k$ singularity at $x$ hence equals
$$ \frac{(r-1)r^{n+n(n+1)/2} + (r^n-1)r^{n(n+1)/2} + p_{n,r}}{r^{1 + n + n(n+1)/2}} = 1 - \frac{r^{n(n+1)/2} - p_{n,r}}{r^{1 + n + n(n+1)/2}}. $$
where $p_{n,r}$ is the number of quadratic forms in $n$ variables of rank $\geq n-1$ over $\kappa(x)$.
The bounds from the subsequent Lemma~\ref{L:quadricrank} give
\begin{align*}
  1 - r^{-n-4} \geq  1 - \frac{r^{n(n+1)/2} - p_{n,r}}{r^{1 + n + n(n+1)/2}} \geq 1 - r^{-n-3}. & \qedhere
\end{align*}
\end{proof}

\begin{lemma}\label{L:quadricrank}
 The number $p_{n,q}$ of quadratic forms in $n$ variables of rank $\geq n-1$ over a field with $q$ elements equals
 $$ p_{n,q} = \prod_{i=1}^{\lfloor (n-1)/2 \rfloor} \frac{q^{2i}}{q^{2i}-1} \prod_{i=0}^{n-2} (q^{n-i}-1) - \prod_{i=1}^{\lfloor n/2 \rfloor} \frac{q^{2i}}{q^{2i}-1} \prod_{i=0}^{n-1} (q^{n-i}-1).$$
 Moreover, 
 $$  q^{\frac{n(n+1)}{2}}(1-q^{-2}) \leq p_{n,q} \leq q^{\frac{n(n+1)}{2}}(1-q^{-3}).$$
\end{lemma}
\begin{proof}
 The formula for $p_{n,q}$ can be found in \cite{MacWilliams-Orthogonal}*{Theorem~2}. Suppose first that $n$ is even. Then
 \begin{align*}
p_{n,q} &= \left( 1 + \frac{q^n}{q^n-1} \cdot (q-1)   \right) \prod_{i=1}^{n/2-1} \frac{q^{2i}}{q^{2i}-1} \prod_{i=0}^{n-2} (q^{n-i}-1)\\
&=   \frac{q^{n+1}-1}{q^n-1}  \cdot \prod_{i=1}^{n/2-1} \frac{q^{2i}}{q^{2i}-1} \prod_{i=0}^{n-2} (q^{n-i}-1) \\
&=  \prod_{i=1}^{n/2-1} q^{2i} \cdot \prod_{i=0}^{n/2-1} (q^{n+1-2i}-1) \\
&=   \prod_{i=0}^{n/2-1} (q^{n+1}-q^{2i}) \\
&= q^\frac{n(n+1)}{2} \prod_{i=0}^{n/2-1} (1-q^{2i-n-1}) \\
&=  q^\frac{n(n+1)}{2} (q^{-n-1}; q^2)_{n/2} ,
 \end{align*}
 where we used the notation for the $q$-Pochhammer symbol. It is clear that
 $(q^{-n-1}; q^2)_{n/2}$ is a decreasing sequence bounded above from $1 - q^{-3}$. Induction on $q \geq 2$ shows the inequality
 $$ \prod_{i=3}^{n} (1 - q^{-i}) \geq 1 - q^{-2} + q^{-n}, $$
 whence $$(q^{-n-1}; q^2)_{n/2} \geq \prod_{i=3}^\infty (1-q^{-3}) \geq 1 - q^{-2}.$$
 For odd $n$, we can reduce to the even case by observing that $p_{n,q} = q^n \cdot p_{n-1,q}$.
\end{proof}

The following proves Theorem~\ref{T:Density}:
\begin{corollary}\label{C:Density} Let $K$ be a finite field of odd characteristic. Then
\begin{align*}
  \lim_{d \to \infty} \frac{\#\{f \in K[x_0,\dots,x_n]_d \mid \{f = 0\} \subseteq \mathbb P^n_{K} \text{ has no defect}\}}{\#K[x_0,\dots,x_n]_d}  \geq \frac{1}{\zeta_{\mathbb P^n_K}(n+3)}.
 \end{align*} 
\end{corollary}
\begin{proof}
 For a property $P$ of hypersurfaces defined by polynomials $f \in K[x_0,\dots,x_n]_d$, write 
 $$\mu(P) := \lim_{d \to \infty} \frac{\#\{f \in K[x_0,\dots,x_n]_d \mid \{f = 0\} \subseteq \mathbb P^n_{K} \text{ satisfies } P\}}{\#K[x_0,\dots,x_n]_d}.$$
 By Theorems \ref{T:MainRes} and \ref{T:DensityMilnor}, there is a constant $c > 0$ such that
 $$ \mu(\text{defect and at most $A_k$ singularities}) \leq \mu(\tau(f) > c \cdot d) = 0.$$
 Moreover, combining Lemma~\ref{L:AkDensity} with \cite{Poonen-Bertini}*{Theorem~1.3},
 \begin{align*} \mu(\text{defect and worse than $A_k$ singularities}) &\leq \mu(\text{worse than $A_k$ singularities}) \\
 &\leq 1 - \frac{1}{\zeta_{\mathbb P^n_K}(n+3)}.
 \end{align*}
 Putting this together,
 \begin{align*}
\mu(\text{no defect}) &= 1 - \mu(\text{defect}) \\
&= 1 - \mu(\text{defect and at most $A_k$ sing.})  - \mu(\text{defect and worse than $A_k$ sing.}) \\
&\geq \frac{1}{\zeta_{\mathbb P^n_K}(n+3)},
 \end{align*}
which completes the proof.
\end{proof}

\begin{remark}
In view of Theorem~\ref{T:MainRes}, we could have added the contribution of ordinary multiple points. The probability for a hypersurface to have a singularity at a point $x$ and this being an ordinary multiple point of multiplicity $\geq 3$, equals
$$ \sum_{d \geq 3} \#\{ f \in \kappa(x)[x_1,\dots,x_n]_d \mid \{f = 0\} \text{ is smooth} \} \cdot \#\kappa(x)^{-\binom{n+d}{d}}.$$
This turns out to be small compared to the local density of at most $A_k$ singularities and we do not expect this to bring a substantial improvement to the bound given in Lemma~\ref{L:AkDensity}.
\end{remark}

\bibliography{references}

\vspace{2\baselineskip}

\begin{minipage}{0.48\textwidth}
\footnotesize
\textit{Address:}\\
{\scshape
Institut f\"ur Algebraische Geometrie\\
Leibniz Universität Hannover\\
Welfengarten 1\\
30167 Hannover\\
Germany}
\\[\baselineskip]
\textit{E-mail:} \texttt{lindnern@math.hu-berlin.de}
\end{minipage}
\begin{minipage}{0.48\textwidth}
\footnotesize
\textit{Current address:}\\
{\scshape
Institut f\"ur Mathematik\\
Humboldt-Universität zu Berlin\\
Unter den Linden 6\\
10099 Berlin\\
Germany}
\\[\baselineskip]
\end{minipage}

\end{document}